\documentclass[12pt]{amsart}

\usepackage[utf8]{inputenc}  

\usepackage[T1]{fontenc}
\usepackage[british]{babel}
\usepackage[style=american]{csquotes}

\usepackage{geometry}
\usepackage{amsthm}
\usepackage{amsfonts}
\usepackage{amsmath}
\usepackage{amssymb}
\usepackage{mathtools} 
\usepackage{eucal}

\usepackage[proportional]{libertine}
\usepackage[libertine]{newtxmath}
\useosf 
\usepackage[scaled=0.83]{beramono}
\usepackage[frak=euler,cal=stixplain]{mathalpha}

\makeatletter
\DeclareMathAlphabet{\mathbfit}{\tx@enc}{\rmdefaultB}{b}{it}
\makeatother

\usepackage{microtype}

\usepackage[style=alphabetic,datecirca=true,dateuncertain=true]{biblatex}
\usepackage[simplepostnote]{lumsdaine-biblatex-misc}

\DeclareCiteCommand{\concisecite}
  [\mkbibparens]
  {\usebibmacro{prenote}} 
  {\printtext{%
     \printnames[family]{labelname}%
     \setunit{\addcomma\space}}%
   \printtext{%
     \printfield{title}%
     \setunit{\addspace}%
     \printtext[parens]{%
       \printdate}}}
  {\multicitedelim}
  {\usebibmacro{postnote}} 

\usepackage{xcolor}
\definecolor{darkgreen}{rgb}{0,0.45,0}
\definecolor{darkred}{rgb}{0.75,0,0}
\definecolor{darkblue}{rgb}{0,0,0.6}
\usepackage[colorlinks,citecolor=darkblue,linkcolor=darkblue,urlcolor=darkblue]{hyperref}
\usepackage{zref-clever}
\zcsetup{cap,sort=false,nameinlink}
\newcommand{\cref}[1]{\zcref{#1}}
\newcommand{\Cref}[1]{\zcref[S]{#1}}

\newcounter{mastertheoremcounter}[section]

\newcommand{\newzctheorem}[2]{
  \newtheorem{#1}[mastertheoremcounter]{#2}
    \AddToHook{env/#1/begin}{%
      \zcsetup{countertype={mastertheoremcounter=#1}}}
}

\usepackage[shortlabels]{enumitem}

\makeatletter
\RenewDocumentCommand{\item}{o}{%
  \@inmatherr\item
  \IfNoValueTF{#1}{%
    \@noitemargtrue\@item[\@itemlabel]%
  }{%
    \def\@currentlabel{#1}\@item[#1]\MakeLinkTarget{}%
  }%
}
\makeatother

\usepackage{mathpartir}

\usepackage{graphicx}
\usepackage{tikz}
\usepackage{tikz-cd}

\usepackage{xifthen} 
\usepackage[normalem]{ulem} 
\usepackage[en-GB]{datetime2}
\usepackage{needspace}

\theoremstyle{plain}
\newzctheorem{theorem}{Theorem}
\newzctheorem{proposition}{Proposition}
\newzctheorem{lemma}{Lemma}
\newzctheorem{fact}{Fact}
\newzctheorem{corollary}{Corollary}
\newzctheorem{conjecture}{Conjecture}

\theoremstyle{definition}
\newzctheorem{definition}{Definition}
\newzctheorem{remark}{Remark}
\newzctheorem{convention}{Convention}
\newzctheorem{notation}{Notation}
\newzctheorem{example}{Example}
\newzctheorem{examples}{Examples}
\newzctheorem{question}{Question}
\newzctheorem{caveat}{Caveat}

\zcRefTypeSetup{convention}{
  Name-sg = Convention ,
  name-sg = convention ,
  Name-pl = Conventions ,
  name-pl = conventions ,
  Name-sg-ab = Conv. ,
  name-sg-ab = conv. ,
  Name-pl-ab = Convs. ,
  name-pl-ab = convs. ,
}
\zcRefTypeSetup{caveat}{
  Name-sg = Caveat ,
  name-sg = caveat ,
  Name-pl = Caveats ,
  name-pl = caveats ,
}


\newif\ifshowprivatenotes
\showprivatenotestrue
\AtBeginDocument{\ifshowprivatenotes \GenericWarning{}{Private notes visible.} \fi}
\newcommand{\hideprivate}{\showprivatenotesfalse}

\newcommand{\privatenote}[1]{\ifshowprivatenotes #1 \else \ignorespaces \fi}
\NewDocumentEnvironment{private}{+b}{\ifshowprivatenotes#1\fi}{}

\newenvironment{colored}[1]{\begingroup \color{#1}}{\endgroup}
\newcommand{\coloredprivatenote}[2]{\privatenote{\textcolor{#1}{#2}}}

\colorlet{todocolor}{darkred}
\colorlet{outdatedcolor}{darkblue!70!black!40}
\colorlet{commentcolor}{darkgreen}

\newcommand{\todo}[1]{\coloredprivatenote{todocolor}{#1}}

\newcommand{\comment}[1]{\coloredprivatenote{commentcolor}{#1}}

\NewDocumentEnvironment{longtodo}{+b}%
  {\begin{private}\begin{colored}{todocolor} #1 \end{colored} \end{private}}
  {}
\NewDocumentEnvironment{longoutdated}{+b}%
  {\begin{private}\begin{colored}{outdatedcolor} #1 \end{colored} \end{private}}
  {}
\NewDocumentEnvironment{longcomment}{+b}%
  {\begin{private}\begin{colored}{commentcolor} #1 \end{colored} \end{private}}
  {}

\newcommand{\defemph}[1]{\emph{#1}} 

\newcounter{saveenumi}

\usetikzlibrary{calc}
\usetikzlibrary{fit}
\usetikzlibrary{shapes}
\usetikzlibrary{arrows}
\usetikzlibrary{decorations.markings}
\usetikzlibrary{decorations.pathmorphing}
\usetikzlibrary{positioning}
\usetikzlibrary{intersections}

\tikzset{
  commutative diagrams/arrow style=tikz,
  commutative diagrams/diagrams={row sep=large},
}

\tikzset{cd-style/.style={commutative diagrams/every diagram}}
\tikzset{cd-arrow-style/.style={commutative diagrams/.cd, every arrow, every label}}
\newcommand{\arr}[1][]{\draw[cd-arrow-style,#1]}

\makeatletter
\tikzset{
  label/.style n args={2}{
    edge node={node [
      execute at begin node=\iftikzcd@mathmode$\fi,
      execute at end node=\iftikzcd@mathmode$\fi,
      /tikz/commutative diagrams/.cd,every label,
      #2
      ] {#1}}
  },
  fine-description/.style={
    auto=false,
    inner xsep=0.3ex,
    inner ysep=0pt,
    fill=white
  }
}
\makeatother

\tikzset{fibtip/.tip={Triangle[open,angle=60:4.5pt]}}
\tikzset{tfibtip/.tip={Bar[sep]Triangle[open,angle=45:4pt]}} 

\pgfarrowsdeclaredouble{fibbtip}{fibbtip}{fibtip}{.fibtip}

\pgfarrowsdeclarealias{c}{c}{left hook}{right hook}
\pgfarrowsdeclarealias{c'}{c'}{right hook}{left hook}

\tikzset{sup/.style={label={#1}{auto=left,pos=1}}}  
\tikzset{sub/.style={label={#1}{auto=right,pos=1}}}

\tikzset{fib/.code={\pgfsetarrowsend{fibtip}}}
\tikzset{fibb/.code={\pgfsetarrowsend{fibbtip}}}
\tikzset{tfib/.code={\pgfsetarrowsend{tfibtip}}}
\tikzset{inj/.code={\pgfsetarrowsstart{c}}}
\tikzset{inj'/.code={\pgfsetarrowsstart{c'}}}
\tikzset{cover/.style={->>}}
\tikzset{dec/.style={sub={\mathrm{dec}}}}
\tikzset{strenum/.style={}}
\tikzset{tcof/.style={tail}}
\tikzset{zigzag/.style={commutative diagrams/rightsquigarrow}}
\tikzset{lw/.style={"lw"{#1}},lw/.default={very near end}}
\tikzset{lw'/.style={"lw"'{#1}},lw'/.default={very near end}}
\tikzset{weq/.style={"\sim"}}
\tikzset{weq'/.style={"\sim"'}}

\newcommand{\generalto}[2]{ \mathrel{\mkern-1mu
  \tikz[baseline={([yshift=-0.58ex]a.south)}]{%
    \node[minimum width=1em,align=center,inner xsep=0.5ex,inner ysep=0.15ex] (a) {$\scriptstyle #2$};
    \arr[#1] (a.south west) -- (a.south east);}
 \mkern-1mu}}

\newcommand{\generalfrom}[2]{ \mathrel{\mkern-1mu
  \tikz[baseline={([yshift=-0.58ex]a.south)}]{%
    \node[minimum width=1em,align=center,inner xsep=0.5ex,inner ysep=0.15ex] (a) {$\scriptstyle #2$};
    \arr[#1] (a.south east) -- (a.south west);}
  \mkern-1mu}}

\renewcommand{\to}[1][]{ \generalto{->}{#1} }
\newcommand{\from}[1][]{ \generalfrom{->}{#1} }

\newcommand{\injto}[1][]{ \generalto{->,inj}{#1} }

\newcommand{\coverto}[1][]{ \generalto{->>}{#1\ \,} }

\makeatletter
\newcommand{\weqto}[1][]{\@ifmtarg{#1}{\generalto{->}{\sim}}{\NotYetDefined}}
\newcommand{\weqfrom}[1][]{\@ifmtarg{#1}{\generalfrom{->}{\sim}}{\NotYetDefined}}
\makeatother
\newcommand{\zigzagto}[1][]{ \generalto{zigzag}{#1} }

 \NewDocumentCommand{\parpair}{O{1.6em}mO{}O{}mO{}O{}}{%
  \mathrel{\mkern-1mu
  \tikz[baseline={([yshift=-0.58ex]a.south)}]{%
    \node[minimum width={#1},inner xsep=0.5ex,inner ysep=0.15ex] (a) {$ $};
    \arr[#4] ([yshift= 0.35ex]a.south west) \IfBlankTF{#2}{to}{to[label={#2}{fine-description,yshift=0.15ex,#3}]} ([yshift= 0.35ex]a.south east);
    \arr[#7] ([yshift=-0.35ex]a.south west) \IfBlankTF{#5}{to}{to[label={#5}{fine-description,yshift=-0.1ex,#6}]} ([yshift=-0.35ex]a.south east);
    }
 \mkern-1mu}}

\tikzset{drpb/.style={commutative diagrams/.cd, phantom, "\lrcorner", very near start}}

\ExplSyntaxOn
\NewDocumentCommand{\mathcalit}{m}
{
  \tl_map_function:nN {#1} \mathcalit_char:n
}
\cs_new_protected:Nn \mathcalit_char:n
{
  \str_if_eq:eeTF {#1} {\str_lowercase:n {#1}}
  { \mathit{#1} } 
  { \mathcal{#1} } 
}

\NewDocumentCommand{\mathbfcalit}{m}
{
  \tl_map_function:nN {#1} \mathbfcalit_char:n
}
\cs_new_protected:Nn \mathbfcalit_char:n
{
  \str_if_eq:eeTF {#1} {\str_lowercase:n {#1}}
  { \mathbfit{#1} } 
  { \mathbfcal{#1} } 
}

\ExplSyntaxOff

\newcommand{\ccK}{\mathcal{K}} 

\newcommand{\ordcatletter}[1]{\mathcal{#1}}
\newcommand{\cB}{\ordcatletter{B}}
\newcommand{\cC}{\ordcatletter{C}}
\newcommand{\cD}{\ordcatletter{D}}
\newcommand{\cE}{\ordcatletter{E}}
\newcommand{\cF}{\ordcatletter{F}}
\newcommand{\cG}{\ordcatletter{G}}
\newcommand{\cI}{\ordcatletter{I}}
\newcommand{\cM}{\ordcatletter{M}}

\newcommand{\cX}{\ordcatletter{X}}

\newcommand{\indcatletter}[1]{\mathbfcal{#1}}
\newcommand{\indC}{\indcatletter{C}}
\newcommand{\indD}{\indcatletter{D}}

\newcommand{\indU}{\indcatletter{U}}

\newcommand{\smallcatletter}[1]{\mathrm{#1}}
\newcommand{\smC}{\smallcatletter{C}}
\newcommand{\smD}{\smallcatletter{D}}
\newcommand{\smE}{\smallcatletter{E}}

\newcommand{\smU}{\smallcatletter{U}}
\newcommand{\smG}{\smallcatletter{G}}

\newcommand{\Gbar}{\overline{G}}
\newcommand{\smGbar}{\overline{\smG}}

\newcommand{\thT}{\mathsf{T}}

\newcommand{\skS}{\mathbf{S}}

\newcommand{\pow}{P}
\newcommand{\syncat}[1]{\smC_{#1}} 
\newcommand{\N}{\mathbb{N}} 
\newcommand{\Ub}{U} 
\newcommand{\Ut}{\tilde{U}} 

\newcommand{\hugetwocatfont}[1]{\mathcal{\uppercase{#1}}}
\newcommand{\twocatfont}[1]{\mathcalit{#1}}
\newcommand{\catfont}[1]{\mathrm{#1}}

\makeatletter
\newcommand{\hugeTwoCatGeneral}[2]{\hugetwocatfont{#1}\@ifnotmtarg{#2}{_{#2}}}
\newcommand{\CAT}[1][]{\hugeTwoCatGeneral{C{\mkern-1.5mu}AT\@ifnotmtarg{#1}{\mkern-3.5mu}}{#1}}
\newcommand{\LFP}[1][]{\hugeTwoCatGeneral{LFP}{#1}}
\newcommand{\REG}[1][]{\hugeTwoCatGeneral{REG}{#1}}
\newcommand{\EX}[1][]{\hugeTwoCatGeneral{E{\mkern-1mu}X}{#1}}
\newcommand{\LEX}[1][]{\hugeTwoCatGeneral{LE{\mkern-1mu}X}{#1}}
\newcommand{\LEXT}[1][]{\hugeTwoCatGeneral{LE{\mkern-1mu}XT}{#1}}
\newcommand{\REGLEXT}[1][]{\hugeTwoCatGeneral{REGLE{\mkern-1mu}XT}{#1}}
\newcommand{\PRETOP}[1][]{\hugeTwoCatGeneral{PRET{\mkern-3mu}OP}{#1}}
\newcommand{\LOCOS}[1][]{\hugeTwoCatGeneral{LOC}{#1}}
\newcommand{\REGLOC}[1][]{\hugeTwoCatGeneral{REGLOC}{#1}}
\newcommand{\AU}[1][]{\hugeTwoCatGeneral{AU}{#1}}

\NewDocumentCommand{\LOG}{O{}O{}}
  {\@ifnotmtarg{#1}{#1\text{-}}\hugeTwoCatGeneral{LOG}{#2}}

\newcommand{\twoCatGeneral}[2]{\twocatfont{#1}\@ifnotmtarg{#2}{_{#2}}}
\newcommand{\Cat}[1][]{\twoCatGeneral{C{\mkern-2mu}at}{#1}}
\newcommand{\Lex}[1][]{\twoCatGeneral{L{\mkern-2mu}ex}{#1}}
\newcommand{\Reg}[1][]{\twoCatGeneral{R{\mkern-2mu}eg}{#1}}
\newcommand{\Log}[1][]{\@ifnotmtarg{#1}{#1\text{-}}\twocatfont{L{\mkern-2mu}og}}

\newcommand{\intCatGeneral}[2]{\catfont{#1}\@ifnotmtarg{#2}{(#2)}}
\newcommand{\intCat}[1][]{\intCatGeneral{Cat}{#1}} 
\newcommand{\intReg}[1][]{\intCatGeneral{Reg}{#1}}
\NewDocumentCommand{\intLog}{O{}O{}}
  {\@ifnotmtarg{#1}{#1\textrm{-}}\intCatGeneral{Log}{#2}}

\newcommand{\strCatGeneral}[2]{\catfont{#1}_\str\@ifnotmtarg{#2}{(#2)}}
\newcommand{\strCat}[1][]{\strCatGeneral{Cat}{#1}}

\newcommand{\strReg}[1][]{\strCatGeneral{Reg}{#1}}
\NewDocumentCommand{\strLog}{O{}O{}}
  {\@ifnotmtarg{#1}{#1\textrm{-}}\strCatGeneral{Log}{#2}}
\newcommand{\Set}{\catfont{Set}}

\newcommand{\Mod}[2][]{\catfont{Mod}\@ifnotmtarg{#1}{_{#1}}(#2)} 

\newcommand{\@F}{\mathcal{F}} 
\NewDocumentCommand{\freeTop}{ O{} O{}}
  {\@ifmtarg{#1}
    {\@ifmtarg{#2}
      {\@F}
      {\@F^{#2}}
    }
    {\@ifmtarg{#2}
      {\@F_{#1}}
      {\@F^{#2}_{#1}}
    }
  }
\makeatother

\newcommand{\colim}{\varinjlim}
\renewcommand{\lim}{\varprojlim}

\newcommand{\comma}[2]{({#1} \mathbin{\downarrow} {#2})} 
\newcommand{\isocomma}[2]{({#1} \mathbin{\downarrow_{\iso}} {#2})} 

\newcommand{\extcat}[1]{\left[ #1 \right]} 

\newcommand{\toposcore}[1]{{#1}^\mathrm{rk}} 

\newcommand{\restrictindexed}[2]{#1{\restriction_{#2}}} 

\newcommand{\indhomwithcatsubscript}[5]{#2_{#3}(#4,#5)_{#1}}
\newcommand{\indhomwithunderline}[5]{\underline{#2_{#3}}(#4,#5)}
\newcommand{\indhomwithHom}[5]{\underline{\operatorname{Hom}}_{#3}(#4,#5)}
\newcommand{\indhomwithboldparens}[5]{#2_{#3} \boldsymbol{(}#4,#5\boldsymbol{)}}
\newcommand{\indhom}[5][]{\indhomwithcatsubscript{#1}{#2}{#3}{#4}{#5}}

\newcommand{\indyon}[1]{\mathbf{y}_{#1}} 

\newcommand{\NNO}{N} 
\newcommand{\List}{\mathop{\mathit{List}}} 

\newcommand{\Pred}{\operatorname{Pred}} 

\newcommand{\op}{\mathrm{op}}
\DeclareMathOperator{\ob}{ob}
\DeclareMathOperator{\mor}{mor}

\newcommand{\id}{\mathrm{id}}

\newcommand{\full}{\mathrm{full}}
\newcommand{\exlex}[1]{{#1}_{\mathrm{ex/lex}}}
\newcommand{\exreg}[1]{{#1}_{\mathrm{ex/reg}}}
\newcommand{\reglex}[1]{{#1}_{\mathrm{reg/lex}}}
\DeclareMathOperator{\Numeral}{num}

\newcommand{\str}{\mathit{str}}

\newcommand{\defeq}{\coloneqq}

\newcommand{\Iff}{\Leftrightarrow}
\newcommand{\Imp}{\Rightarrow}

\newcommand{\iso}{\cong}
\renewcommand{\equiv}{\simeq}

\newcommand{\functorin}[1]{\mathrel{\raisebox{0.065ex}{\textsc{\lowercase{#1}}}\mkern-0.5mu{\in}}}

\newcommand{\bigland}{\bigwedge}

\newcommand{\proves}{\vdash}

\newcommand{\of}[1][2]{\mspace{#1 mu plus 1mu minus 1mu}\mathord{:}\mspace{#1 mu plus 1mu minus 1mu}}

\newcommand{\lforall}[2][\ ]{\forall\mkern1mu #2 #1}
\newcommand{\lexists}[2][\ ]{\exists\mkern2mu #2 #1}

\newcommand{\theoryacronymfont}[1]{\mathrm{#1}}
\newcommand{\internalsyntaxfont}[1]{\mathsf{#1}}
\makeatletter
\newcommand{\IHOL}[1][]{\ensuremath{\theoryacronymfont{IHOL}_{\NNO}\@ifnotmtarg{#1}{^{#1}}}}
\newcommand{\IHOLint}[1][]{\ensuremath{\internalsyntaxfont{IHOL}_{\internalsyntaxfont{N}}\@ifnotmtarg{#1}{^{#1}}}}
\newcommand{\HAS}{\ensuremath{\theoryacronymfont{HAS}}}
\newcommand{\HAH}{\ensuremath{\theoryacronymfont{HAH}}}
\newcommand{\HAomega}{\ensuremath{\theoryacronymfont{HA}^\omega}}
\makeatother

\newcommand{\rulefont}[1]{\mathrm{#1}}
\newcommand{\RCC}{\rulefont{RC}_{\omega}}
\newcommand{\RDC}{\rulefont{RC}_{\rulefont{dep}}}
\newcommand{\RCFT}{\rulefont{RC}_{\rulefont{FT}}}
\newcommand{\EP}{\rulefont{EP}}
\newcommand{\DP}{\rulefont{DP}}

\newcommand{\suchthat}{\mathrel{\mid}}

\newcommand{\lscott}{\mathopen{[\mkern-4mu[}}
\newcommand{\rscott}{\mathclose{]\mkern-4mu]}}

\makeatletter
\NewDocumentCommand{\HOLinterp}{ O{-} O{}}
  {\lscott\, #1 \,\rscott\@ifnotmtarg{#2}{^{#2}}}
\NewDocumentCommand{\AUinterp}{ O{-} O{}}
  {\lscott\, #1 \,\rscott\@ifnotmtarg{#2}{^{#2}}}
\NewDocumentCommand{\EATinterp}{ O{-} O{}}
  {\left[\, #1 \,\right]\@ifnotmtarg{#2}{^{#2}}}
\NewDocumentCommand{\FormalContext}{ O{-} O{}}
  {\left\{\, #1 \,\right\}\@ifnotmtarg{#2}{^{#2}}}
\makeatother

\DeclareFontFamily{OMX}{MnSymbolE}{}
\DeclareSymbolFont{MnLargeSymbols}{OMX}{MnSymbolE}{m}{n}
\SetSymbolFont{MnLargeSymbols}{bold}{OMX}{MnSymbolE}{b}{n}
\DeclareFontShape{OMX}{MnSymbolE}{m}{n}{
    <-6>  MnSymbolE5
   <6-7>  MnSymbolE6
   <7-8>  MnSymbolE7
   <8-9>  MnSymbolE8
   <9-10> MnSymbolE9
  <10-12> MnSymbolE10
  <12->   MnSymbolE12
}{}
\DeclareFontShape{OMX}{MnSymbolE}{b}{n}{
    <-6>  MnSymbolE-Bold5
   <6-7>  MnSymbolE-Bold6
   <7-8>  MnSymbolE-Bold7
   <8-9>  MnSymbolE-Bold8
   <9-10> MnSymbolE-Bold9
  <10-12> MnSymbolE-Bold10
  <12->   MnSymbolE-Bold12
}{}
\DeclareMathDelimiter{\ulcorner}
    {\mathopen}{MnLargeSymbols}{'036}{MnLargeSymbols}{'036}
\DeclareMathDelimiter{\urcorner}
    {\mathclose}{MnLargeSymbols}{'043}{MnLargeSymbols}{'043}

\newcommand{\mathquote}[1]{\left\ulcorner \kern-0.1em #1 \kern-0.1em \right\urcorner}

\title[Projectivity of $\NNO$ in the free topos]{Makkai's lost proof of projectivity of $\NNO$ \\ in the free topos}

\author[H.~Forssell]{Henrik Forssell}
\address{Institutt for informatikk\\ Universitetet i Oslo\\ Oslo, Norway}

\author[P.~LeF.~Lumsdaine]{Peter LeFanu Lumsdaine}
\address{Matematiska institutionen\\ Stockholms universitet\\Stockholm, Sweden}

\author[A.~W.~Swan]{Andrew W Swan}
\address{Fakulteta za matematiko in fiziko\\Univerza v Ljubljani\\Ljubljana, Slovenia}

\IfFileExists{freeze-date.tex}%
  {\date{\DTMdisplaydate{2026}{05}{01}{5}%
}}
  {\date{\textcolor{todocolor}{\today}}\GenericWarning{}{Date not frozen, using today}}

\hideprivate 

\begin{document}

\begin{abstract}
  We give a categorical proof of the projectivity of $\NNO$ in the free topos
  --- in proof-theoretic terms, the rule of countable choice for intuitionistic higher-order logic ---
  based on the unpublished proof of Michael Makkai \cite{makkai:n-is-projective}.
  
  The presentation aims to be self-contained and accessible to any reader acquainted with elementary toposes and their logic.
  
  \comment{Key technical tools include \emph{ranked toposes}, corresponding to the finite-order fragments of higher-order logic; \emph{arithmetic universes}, for the internalisation of syntax/free structured categories; and \emph{indexed categories}, for handling the interplay of externalisation/internalisation.} \comment{Cut this para?}
\end{abstract}

\maketitle

\setcounter{secnumdepth}{2}
\setcounter{tocdepth}{1}
\begin{private}
\setcounter{tocdepth}{2}
\end{private}

\tableofcontents

\section*{Introduction}

In 1986, Lambek and Scott wrote \cite[Hist.\ comments on \textsection II.20]{lambek-scott:intro}:

\begin{quote}
  The question in the text concerning the projectivity of N in the free topos, equivalently, the countable rule of choice in pure type theory, has of course been solved by proof-theoretical techniques (see Friedman and Scedrov 1983).
  Makkai in unpublished notes (1980) cast the logician's proof into a categorical (actually $2$-categorical) mould.
\end{quote}
Unfortunately, neither of the cited proofs ever saw the light of day, and the proofs have generally been considered lost \cite{van-oosten:crosilla-schuster-review}.

In December 2015, Michael Makkai visited Stockholm, where the first two authors were then postdocs.
We pressed him on this topic; once back home, he unearthed his original (marvellously lucid) notes and kindly gave us his blessing to write up the proof for eventual publication.
After barely another decade, here it is.

In the end, our treatment diverges substantially from Makkai’s original ---
primarily due to the subsequent development of the field, partly of course just from personal taste, and finally from the desire to give as self-contained and elementary an exposition as possible.
In particular, we hope this should be accessible for essentially anyone familiar with elementary toposes and a little of their basic development.

\subsection*{Proof sketch}

Like many constructive systems, \defemph{intuitionistic higher-order logic} ($\IHOL$), the logic of elementary toposes, enjoys various meta-theorems saying that existence proofs can be distilled into constructions.
Most fundamental is the \defemph{existence property} ($\EP$): if $\IHOL$ proves \enquote{$\lexists{x \of A}{\varphi(x)}$}, then there is some $\IHOL$-definable $a \of A$ for which $\IHOL$ proves \enquote{$\varphi(a)$}.

Categorically, this says that in the free topos, every cover of the terminal object $\HOLinterp[ x \of A \suchthat \varphi(x)] \coverto 1$ splits; that is, $1$ is projective.
Freyd’s famous gluing argument gives a very clean, purely categorical proof of this fact.

Our main goal is to strengthen this, to show that the \defemph{rule of countable choice} ($\RCC$) is admissible for $\IHOL$.
Concretely, if $\IHOL$ proves \enquote{$\lforall{n \of N} \lexists{x \of X} \varphi(n,x)$}, then there is some $\IHOL$-definable function $f : N \to X$ for which $\IHOL$ proves \enquote{$\lforall{n \of N}\varphi(n,f(n))$}.
Categorically again, this says that in the free topos, every cover of the natural numbers object $\NNO$ splits; in other words, $\NNO$ is projective.

Certainly, if $\IHOL$ proves \enquote{$\lforall{n} \lexists{x} \varphi(n,x)$}, the existence property tells us that for each (actual) natural number $n$, there is some IHOL-definable $a$ such that $\IHOL$ proves \enquote{$\varphi(\bar{n},a)$}, where $\bar{n}$ is the $n$th numeral.

It is not hard to leverage this a bit further, to an operation instead of an existence statement:
syntax and proofs can be Gödel-numbered, hence well-ordered, so we can define for each $n$ some specific $a_n$ (Gödel-number-minimal, say) such that $\IHOL$ proves \enquote{$\varphi(\bar{n},a_n)$}.
However, this operation is still defined \emph{externally}, not in $\IHOL$ as desired.

So one might next hope to run the entire above argument not in our ambient meta-theory, but \emph{internally} to $\IHOL$, as an argument about a internal formalisation of $\IHOL$ inside itself — denote this $\IHOLint$ to distinguish it.
Then, if $\IHOL$ proves \enquote{$\lforall{n} \lexists{x} \varphi(n,x)$}, we argue as follows.
First, we internalise that proof to show that $\IHOL$ proves \enquote{$\IHOLint$ proves ‘$\lforall{n} \lexists{x} \varphi(n,x)$’}.
Then, we would like to run the gluing argument and the Gödel-numbering trick inside $\IHOL$, to define an operation $n \mapsto a_n$ such that $\IHOL$ proves \enquote{$\lforall[]{n}$, $\IHOLint$ proves ‘$\varphi(\bar{n},a_n)$’}.
Finally, we externalise these latter proofs to get that in fact $\IHOL$ proves \enquote{$\lforall{n} \varphi(n,a_n)$}.

Unfortunately, that cannot quite go through as written.
The alarm bells ring most loudly in the last step: it attempts to use the inference \enquote{if $\IHOLint$ proves ‘$\varphi(\bar{n},a_n)$’, then $\varphi(n,a_n)$ holds}.
But this is the sort of thing which Tarski’s undefinability of truth tells us cannot be provable or even expressible within $\IHOL$ (presuming it is consistent).
The same problem arises in internalising Freyd’s gluing argument, since it requires interpreting $\IHOLint$ internally into a category involving the ambient sets of $\IHOL$.

However, the attempt can be salvaged.
Instead of attempting to internalise the whole of $\IHOL$, we internalise just its $r$th-order fragment $\IHOL[r]$, for some $r \in \N$, and interpret that.
This takes care, but can be done: $\IHOL[r]$ is restricted enough that one can build, in $\IHOL$, a large enough set-indexed family of sets to interpret it.
This is the \emph{reflection theorem}: $\IHOL$ can construct the standard interpretation of each of its fragments $\IHOL[r]$.

Now we can push through the full argument for $\RCC$.
If $\IHOL$ proves \enquote{$\lforall{n} \lexists{x} \varphi(n,x)$}, then by compactness some finite-rank fragment $\IHOL[r]$ suffices for the proof.
Internalising this, $\IHOL$ proves \enquote{$\IHOLint[r]$ proves ‘$\lforall{n} \lexists{x} \varphi(n,x)$’}.
So we can run the gluing/Gödel-numbering argument for $\IHOLint[r]$ inside $\IHOL$, to define an operation $n \mapsto a_n$ such that $\IHOL$ proves \enquote{$\lforall[]{n}$, $\IHOLint[r]$ proves $\varphi(\bar{n},a_n)$}.
Finally, we hit these with the internal interpretation of $\IHOLint[r]$, and get that $\IHOL$ proves \enquote{$\lforall{n} \varphi(n,\HOLinterp[a_n])$} — a definable witness function for the original universal-existential statement, as desired.

We promised a category-theoretic proof, though — and the above sketch is couched mostly in proof-theoretic language!
Where are the categories?

This is where we admit to a bait-and-switch: 
the proof-theoretic presentation is clearer and more elementary to sketch, but we will not actually carry it through any further.
The categorical version requires a little work to set up, but provides a powerful organising framework that makes the details much easier to carry through rigorously and with good generality.

Recasting the above sketch categorically, we introduce a notion of \emph{$r$-ranked topos} — roughly, a category of sets admitting $r$ iterations of the power-set operation — stratifying the theory of elementary toposes analogously to the finite-rank fragments $\IHOL[r] \subseteq \IHOL$.
We then argue as follows:

Any cover $e : A \coverto \NNO$ in the free topos $\freeTop$ lifts, by compactness, to the free $r$-topos $\freeTop[r]$ for some $r \in \N$.
This then internalises to the \emph{internal} free $r$-topos $\freeTop[r][\freeTop]$ in $\freeTop$.
Topos logic is strong enough to build an internal universe closed under $r$ iterations of power-sets, and so to construct the standard interpretation for the internal free $r$-topos and carry out Freyd’s gluing argument.
It follows that $\freeTop$ has a map from $\NNO$ giving, for each $n$, a global section of the fibre $A_{\bar{n}} \subseteq A$ in $\freeTop[r][\freeTop]$.
Applying to this the internal standard interpretation $\freeTop[r][\freeTop] \to \freeTop$ yields the desired section of $e : A \coverto \NNO$ in $\freeTop$.

\subsection*{Outline}

We start in \zcref{sec:prelims} by collecting a grab-bag of background material as required for later chapters, and fixing notation and terminology for the main notions we work with:
first some general categorical logic,
then some basics of essentially algebraic (aka cartesian) theories,
and lastly the construction of the free topos itself.

In \zcref{sec:ranked-toposes}, we define \emph{ranked toposes}, categorical analogues of the bounded-rank fragments of higher-order logic, and develop their essential properties.

In \zcref{sec:indexed-ranked-toposes}, we then develop the results on \emph{internal} ranked toposes required for the main proof;
we use the framework of \emph{indexed categories} to handle the interaction between external and internal ranked toposes,
and of \emph{locoses/arithmetic universes} to construct and analyse internal free ranked toposes.
The results there on exact completion and free models in arithmetic universes may be of independent interest (\zcref{thm:ex-of-locos-is-au}, \zcref{cor:au-godel-numbering}).

With this machinery in place, \zcref{sec:reflection-results} gives the endgame, putting together the main reflection results for projectivity of $\NNO$ and admissibility of $\RCC$ as sketched above (\zcref{thm:n-is-projective-in-f}, \zcref{cor:ihol-rcc}), along with analogues for \emph{dependent} choice (\zcref{thm:total-graph-branch-in-f}, \zcref{cor:ihol-rdc}).
Finally, we look back over the proof, and discuss possible alternative approaches and related literature.

\subsection*{Acknowledgements}

This article has benefited greatly from discussions with — among others — Phil Scott, Mike Shulman, Steve Vickers, Milly Maietti, Paul Taylor, Andrej Bauer, Davide Perinti, and the regular members of the Stockholm Logic Seminar, particularly Erik Palmgren, Johan Lindberg, Håkon Gylterud, Christian Espíndola, and Sina Hazratpour. \todo{Anyone more? Or cut down a little? Hard to draw the line.}
Above all, we are grateful to Michael Makkai for his original notes and his blessing for preparing this exposition.

During its (long) preparation, the authors have been supported at various points by the Swedish Research Council project grant 2015-03835 (PI Erik Palmgren), Research Council of Norway grant 230525, the Knut and Alice Wallenberg Foundation project “Type
Theory for Mathematics and Computer Science” (PI Thierry Coquand), the Air Force Office of Scientific Research grant numbers FA9550-21-1-0009 and FA9550-21-1-0024\footnote{Any opinions, findings and conclusions are those of the authors and do not necessarily reflect the views of the United States Air Force.} \todo{other long-term grants?} and  the COST Action CA20111 EuroProofNet, funded by European Cooperation in Science and Technology, \href{http://www.cost.eu}{\nolinkurl{www.cost.eu}}

\section{Preliminaries} \label{sec:prelims}

We begin by recalling some background material, tailored to the forms we will require.
\Cref{sec:cat-thy-prelims,sec:eats-prelims} cover general category theory and essentially algebraic theories, respectively;
these serve mostly to fix notation and note a few theorems we require, so the impatient reader may safely skip them.
\Cref{sec:free-topos-prelims} introduces the free topos and Freyd’s gluing argument; this serves additionally as a warm-up for the ranked versions of these in \cref{sec:ranked-toposes,sec:indexed-ranked-toposes}.
Two further topics, indexed categories and arithmetic universes, we leave until they appear later in the story.

\subsection{General category theory} \label{sec:cat-thy-prelims}

We assume some background in general category theory — essentially, familiarity with the notions involved in \emph{elementary toposes} and their basic theory.
The early chapters of any text in topos theory should certainly suffice --- \cite[\textsection A1]{johnstone:elephant}, say.
For everything beyond this, we recall at least in outline the necessary definitions and results.

In most choices of definitions and terminology, we follow the Elephant \cite{johnstone:elephant}.
One peculiarity in particular deserves setting out from the start:

\begin{convention}[{\cite[\textsection A1.2]{johnstone:elephant}}] \label{convention:chosen-structure}
  Whenever we speak of structures satisfying existence conditions (e.g.~categories with limits, essentially surjective functors), we mean these to include a choice of witnesses for the existence (chosen limits, chosen lifts-up-to-isomorphism, and so on).
  We do not however ask morphisms to preserve these choices, unless explicitly specified.
\end{convention}

Assuming the Axiom of Choice, this is of course inconsequential.
Our main motivation is for uniformity with \emph{internal} category theory, where the chosen-structure version is almost always advantageous since it presents structured categories as essentially algebraic notions.

With this in mind, the main classes of logically-structured categories we will use throughout are:
\begin{definition}\label{def:basiccategories}
  A category is: 
  \begin{enumerate}
  \item \defemph{cartesian} (aka \defemph{left exact}, \defemph{lex}, \defemph{finite-limit}, \defemph{finitely complete}) if it has all finite limits;
  \item \defemph{lextensive} if it has finite limits, and finite coproducts, disjoint and stable under pullback;
  \item \defemph{regular} if it has finite limits and image factorisations, and covers are pullback-stable (in the sense of \cref{def:covers} below);
  \item \defemph{exact} if it has finite limits, and pullback-stable effective quotients of equivalence relations;
  \item a \defemph{pretopos} if it is lextensive and exact;
  \item an \defemph{elementary topos with NNO}, or just \defemph{topos}, if it is regular and additionally has power-objects and a natural numbers object\footnote{In most literature, \emph{topos} is defined without NNO by default, but the \emph{free topos} means the free topos-with-NNO. In light of the latter, we take the version with NNO as fundamental.}.
\end{enumerate}
\end{definition}

(A few more notions — locally finitely presentable categories, locoses, arithmetic universes — we will recall as and when we need them.)

Reading these according to \Cref{convention:chosen-structure}, each kind of category comes equipped with operations providing the assumed categorical structure.
So there are always two levels at which a functor $F : \cC \to \cD$ between such categories may preserve the logical structure:
\defemph{weakly} (the default sense), in that limits (colimits, power-objects, etc) retain their universal property under $F$, or equivalently that the canonical comparison maps $F(\lim_i X_i) \to \lim_i FX_i$ are isomorphisms;
or \defemph{strictly}, in that the functor commutes on the nose with the structure operations: $F(\lim_i X_i) = \lim_i FX_i$.

\begin{definition}
  A functor $F : \cC \to \cD$ between cartesian categories (resp.~regular categories, toposes, etc.) is \defemph{cartesian} (resp.~regular, logical, etc.)\ if it (weakly) preserves the assumed structure, and \defemph{strictly cartesian} (resp.~strictly regular, etc.)\ if it preserves the chosen structure on the nose.
\end{definition}

We will often have to deal with the interplay of weakly and strictly structure-preseving functors, along with $2$-categorical issues and sometimes size issues; so we systematically distinguish several ($2$-)categories of these logically-structured categories.

\begin{definition} \label{def:cats-of-cats} \needspace{\baselineskip} 
  We write:
  \begin{enumerate}
  \item $\strReg$, $\Reg$, and $\REG$ respectively for the ($1$-)category of small regular categories and \emph{strictly} regular functors; the $2$-category of small regular categories, regular functors, and natural transformations; the $2$-category of locally small regular categories, regular functors, and natural transformations;
  \item $\strCat$, $\Cat$ and $\CAT$ for the corresponding $1$- and $2$-categories of categories (in this case, there is no difference between strict and non-strict functors, but $\strCat$ still differs from $\Cat$ in omitting $2$-cells);
  \item $\LEX$, $\LEXT$, $\EX$, and so on for the corresponding $2$-categories of cartesian, lextensive, exact, etc.~categories, the corresponding structure-preserving functors, and natural transformations (for these, we never need the strict versions);
  \item $\strLog$, $\Log$ and $\LOG$ for the corresponding $1$- and $2$-categories of toposes, logical functors (strict for $\strLog$), and natural \emph{isomorphisms} (omitted in $\strLog$).
  \end{enumerate}
\end{definition}

Viewing a small regular category (or topos, etc.)\ as an object of $\Reg$ (resp.~$\Log$, etc.), its specific choice of logical structure is inessential, since even isomorphisms in $\Reg$ need not preserve this.
Viewing it as an object of $\strReg$, however, the choice matters: different choices of regular structure on a category may yield non-isomorphic objects of $\strReg$.
We will therefore speak of \emph{strict} regular categories (toposes, etc.)~to emphasise when we view them as objects of the strict category.

\begin{caveat} \label{caveat:strict-cats}
  The strict notions are rather sensitively presentation-depend\-ent:
  for instance, common alternative definitions of toposes such as \enquote{regular + power-objects} and \enquote{cartesian closed + subobject classifier} yield equivalent $2$-categories $\LOG$, but inequivalent $1$-categories $\strLog$.

  For this reason, among others, the strict categories $\strLog$, $\strReg$, and so on are generally thought of as auxiliary technical devices, with the $2$-categories $\LOG$, $\REG$, etc.~as the true objects of study. 
\end{caveat}

\begin{definition}\label{def:covers}
  Following \cite[A1.3.2]{johnstone:elephant}, we call a map a \defemph{cover} if it does not factor through any proper subobject of its target.
  We denote covers diagrammatically by $A \coverto B$.
  Say a cover is \emph{stable} if every pullback of it is a cover.

  A \emph{(stable) image factorisation} of a map is a factorisation as a (stable) cover followed by a monomorphism.
\end{definition}

\begin{proposition}[{\cite[A1.3.2–4]{johnstone:elephant}}]
  In a cartesian category, covers are precisely the extremal epimorphisms.
  In a regular category, covers are precisely the regular epimorphisms.
  \qed
\end{proposition}

\begin{definition}
  An object $P \in \cC$ is \emph{cover-projective} (resp.~\emph{regular-}) if maps from $P$ lift along covers (resp.~regular epimorphisms):
  \[ \begin{tikzcd}
    & Y \ar[d,cover,"e"] \\
    P \ar[ur,dashed,"\exists\, g"] \ar[r,"f"] & X
  \end{tikzcd} \]
\end{definition}

By \emph{projective} we always mean cover-projective, or equivalently regular\nobreakdash-, since we consider projectivity only in regular categories.
In a topos, all epis are regular, so ordinary projectivity (i.e.~with respect to epimorphisms) also coincides.
Lastly, we note the easy characterisation:
\begin{proposition} \label{prop:projective-equivalent-split}
  In a regular category, an object is projective just if every cover of it has a splitting. \qed
\end{proposition}

We will make frequent essential use of \emph{comma} and \emph{iso-comma} categories:

\begin{definition}
   Given categories and functors $\cC_0 \to[F_0] \cD \from[F_1] \cC_1$, the \defemph{comma category} $\comma{F_1}{F_0}$ has as objects triples $(c_1, c_0, \varphi)$, where $c_i \in \cC_i$ and $\varphi : F_1 c_1 \to F_0 c_0$, and maps $(c_1, c_0, \varphi) \to (d_1,d_0,\psi)$ are maps $f_i : c_i \to d_i$ such that $F_0 f_0 \varphi = \psi F_1 f_1$.
  The \defemph{iso-comma} $\isocomma{F_1}{F_0}$ is similar, but restricting to objects $(c_1, c_0, \varphi)$ in which $\varphi$ is an isomorphism.
  In both cases, here shown for the comma category, we have projections $P_i: \comma{F_1}{F_0} \to \cC_i $ and a natural transformation (isomorphism in the iso-comma case) $\alpha:F_1P_1\Rightarrow F_0P_0$ whose component at $(c_1, c_0, \varphi)$ is $\varphi$:
  \[ \begin{tikzcd}
     \comma{F_1}{F_0} \ar[d,"P_0"] \ar[r,"P_1"] & \cC_1  \ar[d,"F_1"] \ar[dl,"\alpha"',Rightarrow, shorten=6mm] \\
     \cC_0  \ar[r,"F_0"] & \cD.
  \end{tikzcd} \]

  When either functor is an identity, we will write e.g.~$\comma{\cD}{F}$ for $\comma{\id_\cD}{F}$.
\end{definition}

\begin{proposition} \label{prop:commas}
  Fix $\cC_0 \to[F_0] \cD \from[F_1] \cC_1$ as above.
  \begin{enumerate}[(1)]
  \item \label{item:comma-regular}
    If $\cC_0$, $\cC_1$, $\cD$, and $F_1$ are all regular, and $F_0$ cartesian, then $\comma{F_1}{F_0}$ is regular, and its projection functors $P_i : \comma{F_1}{F_0} \to \cC_i$ are strictly regular. 
    Moreover, a map $(f,f')$ in $\comma{F_1}{F_0}$ is a cover  if both of its components $f$, $f'$ are.
  \item \label{item:isocomma-topos}
    If $\cC_0$, $\cC_1$, and $\cD$ are regular (resp.~toposes), and both $F_i$ are regular (resp.~logical), then $\isocomma{F_1}{F_0}$ is regular (a topos), and its projection functors are strictly regular (strictly logical).
  \end{enumerate}
\end{proposition}

\comment{Note: We really do need $F_1$ regular not just lex for $\comma{F_1}{F_0}$ to be regular: take $\cC_0$ free reg cat on an object $O$, $\cC_1$ free on a \emph{supported/inhabited} object $U$, and both their global sections functors to $\Set$. Then the map $(\Gamma U \to \Gamma O) \to (\Gamma 1 \to \Gamma 1)$ is a cover, but not reg epi.}

\begin{proof}
  All essentially straightforward.
  On the object components, logical structure is given componentwise (as it must be, for the projections to preserve it strictly).
  On the map component, it amounts mostly to the fact that all the regular structure is automatically (covariantly) functorial, and power-objects are functorial in isomorphisms;
  in case $F_0$ is just cartesian, the map component of the image factorisation does not quite fit this pattern, but is determined uniquely by orthogonality between covers and monos.
\end{proof}

\begin{proposition} \label{prop:comma-is-comma}
  For $\cC_0 \to[F_0] \cD \from[F_1] \cC_1$ in $\REG$,
  \begin{enumerate}[(1),beginpenalty=10000] 
  \item \label{item:functo-to-comma-regular} a functor $G : \cX \to \comma{F_1}{F_0}$ is regular, or strictly so, just if its object components $P_i G : \cX \to \cC_i$ both are;
  \item for regular $G_i : \cX \to \cC_i$, natural transformations $\gamma : F_1 G_1 \to F_0 G_0$ correspond bijectively to regular functors $G : \cX \to \comma{F_1}{F_0}$ lifting $(G_1,G_0) : \cX \to \cC_1 \times \cC_0$ (related via $\alpha G = \gamma$).
  \end{enumerate}
  \[ \begin{tikzcd}[sep=small]
    & \cC_1 \ar[dr,"F_1"] \ar[dd,"\gamma",Rightarrow,dashed,shorten=1.5ex] &
    &[-1em] &[1.5em] &[-1em] 
    &[1.4em] { } \ar[dr,phantom,"="{sloped,near end}] &[-2.3em] &[-0.5em] \cC_1 \ar[dr,"F_1"] \ar[dd,"\alpha",Rightarrow,shorten=1.5ex] &
    \\
    \cX \ar[dr,"G_0"'] \ar[ur,"G_1"] & & \cD
    & { } \ar[r,squiggly,leftrightarrow] & { }
    & \cX \ar[drrr,"G_0"',bend right=10] \ar[urrr,"G_1",bend left=10] \ar[rr,"G",dashed] & & \comma{F_1}{F_0}\mkern-15mu \ar[dr,"P_0"] \ar[ur,"P_1"'] & & \cD
    \\
    & \cC_0 \ar[ur,"F_0"'] & 
    & & &
    & { } \ar[ur,phantom,"="{sloped,near end}] & & \cC_0 \ar[ur,"F_0"']
  \end{tikzcd} \]

  Similarly, $\isocomma{F_1}{F_0}$ enjoys analogous properties with respect to natural \emph{isomorphisms}, in $\REG$ and also $\LOG$.
\end{proposition}

In $2$-categorical terms, these form \emph{(iso-)comma objects} in $\REG$, $\LOG$.

\begin{proof}
  Part \zcref[noname]{item:functo-to-comma-regular} amounts to checking that for each object piece of logical structure on $\comma{F_1}{F_0}$ or $\isocomma{F_1}{F_0}$, the map component is uniquely determined given the object parts and commutativity with the corresponding structure maps.
  The latter two parts follow directly.
\end{proof}

Our final comma construction has something of a different character, and will turn out to have very important consequences:

\begin{proposition}[Artin gluing, {\cite[A2.1.12]{johnstone:elephant}}] \label{prop:artin-gluing-topos}
  Let $F : \cC \to \cD$ be a \emph{cartesian} functor between toposes.
  Then $\comma{\cD}{F}$ is a topos, $P_0 : \comma{\cD}{F} \to \cC$ is strictly logical, and $P_1 : \comma{\cD}{F} \to \cD$ is strictly regular.

  Moreover, this is appropriately functorial: pseudo-commuting squares $\varphi : JF' \iso FI$ with $I$, $J$ logical induce logical functors $\comma{J}{\varphi} : \comma{\cD'}{F'} \to \comma{\cD}{F}$, and this respects identities and composition:
  \[
  \begin{tikzcd}
\cC' \ar[r,"F'"] \ar[d,"I"] & \cD' \ar[d,"J"] \ar[dl,phantom,"\iso"{sloped}]
    \\
    \cC \ar[r,"F"] & \cD
  \end{tikzcd}
  \qquad \zigzagto[\qquad] \quad
  \begin{tikzcd}[row sep=small]
    \comma{\cD'}{F'} \ar[dd,"\comma{J}{\varphi}"] \ar[dr] \ar[drr,bend left=15] \\[-2ex]
    & \cC' \ar[r] \ar[dd] & \cD' \ar[dd]
    \\
    \comma{\cD}{F} \ar[dr] \ar[drr,bend left=15] \\[-2ex]
    & \cC \ar[r] & \cD.
  \end{tikzcd}
  \]
\end{proposition}

\begin{proof}
  The regular structure is component-wise as in \cref{prop:commas}\zcref[noname]{item:comma-regular};
  the NNO is also componentwise, with $\NNO_\cD \to F\NNO_\cC$ induced by $F0_\cC$, $FS_\cC$;
  and power-objects are given by the pullback
  \[ 
  \pow \left( \begin{tikzcd}[
        baseline={([yshift=-axis_height]\tikzcdmatrixname)}]
    D \ar[d,"p"] \\ FC 
  \end{tikzcd} \right)
  = \left( \begin{tikzcd}[sep=small,
        baseline={([yshift=-axis_height]\tikzcdmatrixname)}]
    \bullet \ar[dd] \ar[rr] \ar[ddr,drpb] &[0.6em] &[-0.4em] {\subseteq}_{D} \ar[d,inj]
    \\ & & \pow D \times \pow D \ar[d,"\pi_1"]
    \\ F \pow C \ar[r,"\xi_C"] & \pow F C \ar[r,"p^*"] & \pow D
  \end{tikzcd} \hspace{-10em} \right) \hspace{10em}
  \]
  where $\xi_C$ is the canonical comparison map $FPC \to PFC$, corresponding to $F({\in_C}) \injto FC \times F \pow C$.
  In the internal language, the pullback can be written as $\HOLinterp[ a \of[2.5] FPC,\, b \of[2.5] PD \suchthat \allowbreak b \subseteq \{ y \of D \suchthat p(b) \functorin{\textit{F}} a \} ]$.

  Verification of the functoriality is gruesome but routine.
  %
\end{proof}

\subsection{Essentially algebraic theories} \label{sec:eats-prelims}
By \emph{essentially algebraic theories} we mean, initially and broadly,  that class of theories which correspond to cartesian categories, and whose categories of models in $\Set$ are, accordingly, locally finitely presentable. 
There are several different syntactical characterisations of these, including the \emph{essentially algebraic theories} of \cite{freyd:aspects-of-topoi},  the \emph{cartesian theories} of \cite{johnstone:elephant}, and the \emph{partial Horn theories} of  \cite{palmgren-vickers}. 
Of these, we use  cartesian theories, as defined in \cite[D1.3.4]{johnstone:elephant}, when we need a syntactic perspective. 


%
%

\begin{theorem}[{\cite[\textsection 1.C]{adamek-rosicky}, \cite[D]{johnstone:elephant}}]
  \label{thm:eats-characterisations}
  The following are equivalent, for a category $\cM$:
  \begin{enumerate}[(a)]
    \item \label{item:eat-equivs:lfp}
      $\cM$ is locally finitely presentable, i.e.\ it is cocomplete and has a small, full subcategory  $\cM_{\mathrm{fp}}$ of finitely presented objects such that every object is a directed colimit of objects from  $\cM_{\mathrm{fp}}$ \cite[\textsection 1.A]{adamek-rosicky}.
    \item \label{item:eat-equivs:lex}
      $\cM \equiv \LEX(\smC,\Set)$ for some small cartesian category $\smC$.
    \item \label{item:eat-equivs:cart}
      $\cM \equiv \Mod[\Set]{\thT}$ for some cartesian theory $\thT$ of predicate logic.
    \item \label{item:eat-equivs:sketch}
      $\cM \equiv \Mod[\Set]{\skS}$ for some finite limit sketch $\skS$, in the sense of \cite[D2.1.2(b)]{johnstone:elephant}.
  \end{enumerate}

  Furthermore:
  \begin{enumerate}[(1)]
    \item \label{item:eat-facts-duality}
      In condition \zcref[noname]{item:eat-equivs:lex}, the cartesian category $\smC$ is equivalent to $\cM_{\mathrm{fp}}^\op$; indeed, this underlies a biequivalence $\LFP \equiv_2 \Lex^\op$ (\enquote{Gabriel–Ulmer duality}, \cite[Thm.~6]{adamek-porst:alg-theories-of-quasivarieties}).
      \comment{Is the biequivalence already somewhere in \cite{adamek-rosicky}? Couldn’t find it on a skim; but if so, no need to cite \cite{adamek-porst:alg-theories-of-quasivarieties}.}
    
    \item \label{item:eat-facts-finiteness}
      If $\cM \equiv \Mod[\Set]{\thT}$ for some \emph{finitely axiomatised} cartesian theory as in \zcref[noname]{item:eat-equivs:cart}, then it is equivalent to $\Mod[\Set]{\skS}$ for some \emph{finite} finite limit sketch $\skS$ as in \zcref[noname]{item:eat-equivs:sketch}, and vice versa.
    
    \item \label{item:eat-facts-syncat}
      Given a cartesian theory $\thT$ as in \zcref[noname]{item:eat-equivs:cart}, a cartesian category as in \zcref[noname]{item:eat-equivs:lex} is given by the \emph{syntactic category} $\syncat{\thT}$ of $\thT$, with objects cartesian formulas-in-context $[ \vec x \suchthat \varphi(\vec x) ]$ of $\thT$. \qed
  \end{enumerate}
\end{theorem}

\begin{definition}
  By an \defemph{essentially algebraic theory} $\thT$, we mean agnostically a locally finitely presentable category $\cM_\thT$, or a small cartesian category $\syncat{\thT}$, justified by \zcref[noname]{item:eat-facts-duality} above.
  We say $\thT$ is \defemph{finitely presented} if it admits a presentation by a finite cartesian theory, or equivalently a finite finite limit sketch (sic), in light of \zcref[noname]{item:eat-facts-finiteness} above.
  
  A \defemph{model} of $\thT$ in a cartesian category $\cE$ is just a cartesian functor $M : \syncat{\thT} \to \cE$ (equivalently, a model of in $\cE$ the theory or sketch presenting $\thT$); we denote the category of these by $\Mod[\cE]{\thT}$.

  We call objects of $\syncat{\thT}$ \defemph{types} of $\thT$,
  and their interpretations $\EATinterp[X][M] \in \cE$ under a model $M \in \Mod[\cE]{\thT}$ the types of $M$.
  \comment{“Types” not quite ideal, but no obvious better alternatives: context, formula, sort, …?  Consider this; if changing, make sure to change everywhere.}
\end{definition}

\comment{Convention to keep consistent with: We generally say \emph{initial} object/model/etc, not \emph{free}; but just for ($r$-)toposes, we call them \emph{the free topos}, etc.}

We will typically present essentially algebraic theories by describing their category of models, in such a way that they are evidently models of a suitable cartesian theory.
\comment{Expository issue: We never precisely define Cartesian theories, but we assume readers are happy to recognise them when we see them. I think this probably can work, but wants to be said more clearly somewhere above?}
In particular, the $1$-categories of strict logically-structured categories defined above are all essentially-algebraic:

\begin{proposition} \label{prop:logical-cats-as-eats}
  $\strCat$, $\strReg$, $\strLog$ are each the category of models of some essentially algebraic theory.
\end{proposition}

\begin{proof}
  The theory of categories is standard, and easy to write down.
  It can then be extended by operations for the further assumed structure, and operations/axioms enforcing their defining properties.

  The one subtlety is that the definitions of covers and power-objects quantify over monomorphisms, which are not in general an essentially-algebraic notion.
  Given finite limits, though, they can be defined essentially-algebraically as maps whose kernel pair projections are isomorphisms.

  It is then clear that the models of these theories are categories, regular categories, and toposes respectively, and their homomorphisms precisely the strictly structure-preserving functors.
\end{proof}

We exploit several consequences of essential-algebraicity.
Most ubiquitously, it allows us to uniformly internalise all these notions.

\begin{definition} \label{def:internal-cats}
  Let $\cE$ be a cartesian category.
  An \emph{internal} category (regular category, topos, etc.)\ in $\cE$ means a model of the cartesian theory of categories (regular categories, toposes, etc.)\ in $\cE$.
  Maps of these can be viewed as functors (strictly regular, logical, etc.); we denote the resulting $1$-categories by $\strCat[\cE]$, $\strReg[\cE]$, $\strLog[][\cE]$, etc.
  
  With the evident notion of natural transformation, we denote the $2$-category of internal categories by $\intCat[\cE]$.
  (One may also define $\intReg[\cE]$, $\intLog[][\cE]$, with suitable notions of (weakly) regular/logical internal functors; but the definitions are a little involved and we never need them.)
\end{definition}

Our other main uses of essential-algebraicity are to obtain initial logically-structured categories (and later internalise them, in \cref{sec:internal-free-models}), and to justify appeals to \emph{compactness}.
Categorically, this latter amounts to understanding how initial models interact with filtered colimits of theories --- in our application, $\IHOL$ as the colimit of its finite-rank fragments $\IHOL[r]$ (all in their categorical forms, toposes and $r$-toposes).

In fact it turns out more convenient to give the general compactness result dually, in terms of categories of models:

\begin{lemma} \label{lem:initial-object-of-limit-lfpcat}
  Let $\cI$ be a filtered category, $(\cM_i)_{i \in \cI^\op}$ a diagram (strict or pseudo) of lfp categories and finitary functors, and $M_\infty$ some (bi-)limit for it in both $\CAT$ and $\LFP$.%
  \footnote{In fact the limits in $\CAT$ and $\LFP$ should always coincide; but we have not been able to find a reference for this, and checking that both conditions hold when we apply the lemma is easier than proving the general case.
  \comment{Search further for reference??}}
  Write the functors of the diagram and cone as $U_i : \cM_j \to \cM_i$, for each $i \to j$ in $\cI$, and also for $j = \infty$; and write $A_i$ for the initial object of $\cM_i$.
  
  Then for each $i$, $U_i A_\infty \iso \colim_{j \in i/\cI} U_i A_j$.
\end{lemma}

\begin{proof}
  For each $i$, set $A_{\infty,i} \coloneq \colim_{j \in i/\cI} U_i A_j$.
  Functors $j/\cI \to i/\cI$ between coslices of a filtered category are final; so we have isomorphisms $U_i A_{\infty,j} \iso A_{\infty,i}$ for each $i \to j$ making the objects $A_{\infty,i}$ into a coherent family in the diagram  $(\cM_i)_{i \in \cI^\op}$.
  They thus lift to some object $A_{\infty,\infty} \in \cM_\infty$, with coherent isomorphisms $U_i A_{\infty,\infty} \iso A_{\infty,i}$ for each $i$.
  
  Now it is direct to check that $A_{\infty,\infty}$ is initial in $\cM_{\infty}$. Given any $X \in \cM_\infty$, the maps $A_j \to U_j X$ assemble into maps $A_{\infty,i} \to U_i X$ naturally in $i$, and hence into a map $A_{\infty,\infty} \to X$ in $\cM_\infty$. Uniqueness is similar.
\end{proof}

Gabriel–Ulmer duality tells us we can equivalently view this situation as a colimit of theories, $\thT_\infty = \colim_{i \in \cI} \thT_i$.
In applications, however, the limit condition on categories of models is easier to verify.

\begin{proposition}{\cite[D2.4.9]{johnstone:elephant}}
  Let $\thT$ be an essentially algebraic theory, and $X \in \syncat{\thT}$ some type of $\thT$.
  Then for any filtered colimit of models $M = \colim_{i} M_i$, we have
  $\EATinterp[X][M] \iso \colim_i \EATinterp[X][M_i]$.
\end{proposition}

\begin{proof}
  This is the standard fact that filtered colimits in $\LEX(\syncat{\thT},\Set)$ are computed pointwise: $\LEX(\syncat{\thT},\Set)$ is closed under filtered colimits in $\CAT(\syncat{\thT},\Set)$ since filtered colimits commute with finite limits.
\end{proof}

\begin{corollary} \label{cor:covers-in-filtered-colim} \pushQED{\qed}
  Suppose $\smC_\infty = \colim_{i \in \cI} \smC_i$ is a filtered colimit in $\strReg$, and $(X_i \in \smC_i)_{i \in I \cup \{\infty\}}$ a matching family of objects in the diagram and colimit.
  Then we have a filtered colimit of sets
  \[ \textstyle \{ Y \in \smC_\infty,\, e : Y \coverto X_\infty \} = \colim_{i \in \cI} \{ Y \in \smC_i,\, e : Y \coverto X_i \}. \qedhere \] \popQED
\end{corollary}

\subsection{The free topos} \label{sec:free-topos-prelims}

Various constructions of the free topos exist in the literature.
All agree up to equivalence, as they are bi-initial in the $2$-category $\LOG$ of toposes and logical functors; this is the defining property of the free topos \cite[D4.3.14(a)]{johnstone:elephant}.
(An object $\cC$ of a $2$-category $\ccK$ is \defemph{bi-initial} if each hom-category $\ccK(\cC,\cD)$ is contractible: there is an essentially unique map from $\cC$ to each object of $\ccK$.)

Most developments characterise it moreover up to isomorphism, as a strictly initial object in some $1$-category
of toposes and strictly logical functors, or of higher-order theories and interpretations (e.g.\ \cite[13.1]{lambek-scott:intro}).
As with \cref{caveat:strict-cats}, however, it is worth noting that this is somewhat fragile:
different \enquote{equivalent} definitions of toposes yield equivalent but not isomorphic constructions of the free topos.

We present here one possible construction.
This is not essential --- we could work throughout with an off-the-shelf version, using just the $2$-categorical universal property --- but it forms a useful warmup for the development of free ranked and internal toposes below.

\begin{definition}
   We denote by $\freeTop$ the \defemph{free strict topos}: that is, the initial object of $\strLog$ (which exists since strict toposes are an essentially algebraic notion, \cref{prop:logical-cats-as-eats}).
\end{definition}

The desired $2$-categorical universal property follows purely formally with the aid of the iso-comma construction, \cref{prop:commas}\zcref[noname]{item:isocomma-topos}.

\begin{proposition} \label{prop:free-topos-is-bi-initial}
  The free strict topos $\freeTop$ is moreover bi-initial in $\LOG$.
\end{proposition}

\begin{proof}
  We need to show that each hom-category $\LOG(\freeTop, \cE)$ is contractible.
  
  For inhabitedness, take some choice of topos structure on $\cE$; then the substructure generated by its operations is a small (indeed, countable) subtopos $\smE \subseteq \cE$, into which $\freeTop$ maps by its initiality in $\strLog$.
  
  It remains to show that for any logical $F_0, F_1 : \freeTop \to \cE$, there is a unique natural isomorphism $F_0 \iso F_1$.
  But by \cref{prop:comma-is-comma}, these correspond precisely to strictly logical functors $\freeTop \to \isocomma{F_1}{F_0}$ over $(F_1,F_0) : \freeTop \to \cE \times \cE$,
  unique existence of which is immediate by strict initiality of $\freeTop[r]$.
  \[ \begin{tikzcd}[row sep=normal,column sep=large,
                    baseline=(\tikzcdmatrixname-2-1.base)]
      & \isocomma{F_1}{F_0} \ar[d,"{(P_1,P_0)}"] \ar[r] & \cE^{\iso} \ar[d,"{(s,t)}"] \\
      \freeTop \ar[ur,dashed,bend left=10] \ar[r,"\Delta_{\freeTop}"] & \freeTop \times \freeTop \ar[r,"F_1 \times F_0"] & \cE \times \cE
    \end{tikzcd} \qedhere \]
\end{proof}

\begin{remark}
  The rôle of the iso-comma topos here, and specifically the fact that topos structure is closed only under \emph{iso}-comma categories, illustrates why the $2$-cells of $\LOG$ are taken as just natural isomorphisms:
  with non-invertible transformations included, $\freeTop$ would not be bi-initial.
  
  For regular or cartesian categories (say), since comma categories are available, the analogous argument gives a bi-initial object with respect to all natural transformations. 
  
  At base this comes down to the fact that cartesian/regular structure is all covariantly functorial, whereas toposes include both co- and contra-variant structure.
  \comment{Is there some precise statement along these lines in the literature, or useful discussion? It’s a heuristic I (PLL) have become gradualy aware of over the years but don’t remember any specific sources for.}
\end{remark}

\comment{What is the closest to this approach found in the literature?  All treatments I’ve checked (which did this mean?) use internal-language methods to construct $\freeTop$ and show it’s (bi-)initial.  The construction+initiality here is only superficially different from that approach, but the presentation of bi-initiality is possibly novel? I (PLL) taught it with Mike Shulman in a summer school in 2016; I’m not sure whether one of us had learned it from somewhere, or if we worked it out together there.  How to concisely and appropriately say this and acknowledge Mike?}

We can now give Freyd’s classic categorical proof of the existence property for higher-order logic (\cite[\S 6]{lambek-scott:itt-and-the-free-topos}, \cite[1.(10)3]{freyd-scedrov:categories-allegories}).

\comment{What’s a canonical citation for Freyd gluing, and application to projectivity?  Can’t find an original Freyd paper. It’s in \cite[1.(10)3]{freyd-scedrov:categories-allegories}.  Another option is Lambek–Scott 1980 paper, “Intuitionist type theory and the free topos”, or else Lambek–Scott book.  Where is it in the Elephant?  Look in Moerdijk \cite{moerdijk:glueing-topoi}, it has good early citations.
-- Couldn't find anything better. Elephant promises full treatment in part F}

\begin{proposition}[{Freyd gluing}] \label{prop:freyd-gluing-topos}
  The terminal object of $\freeTop$ is projective.
\end{proposition}

\begin{proof}
  Consider the Artin gluing category $\comma{\Set}{\Gamma}$ of the global sections functor $\Gamma : \freeTop \to \Set$, with projection functors $P_0 : \comma{\Set}{\Gamma} \to \freeTop$, $P_1 : \comma{\Set}{\Gamma} \to \Set$. 
  This is a topos, by \cref{prop:artin-gluing-topos};
  so initiality gives a strictly logical \defemph{interpretation} functor $\HOLinterp : \freeTop \to \comma{\Set}{\Gamma}$, and since $P_0$ is strictly logical, $P_0 \HOLinterp = \id_{\freeTop}$.
  Write $\HOLinterp_1$ for the composite $P_1 \HOLinterp : \freeTop \to \Set$, and $\rho$ for the natural transformation $\HOLinterp_1 \to\Gamma P_0 \HOLinterp = \Gamma$.
  \[ \begin{tikzcd}
    & \comma{\Set}{\Gamma} \ar[d,"P_0"] \ar[r,"P_1"]
    & \Set \ar[d,"\id"] \ar[dl,Rightarrow,"\rho"',shorten=6mm]
    \\ \freeTop[r] \ar[ur,dashed,"\HOLinterp"] \ar[r,"\id"]
     & \freeTop[r] \ar[r,"\Gamma"]
     & \Set 
    \end{tikzcd}
  \]

  Now for any cover $e : A \coverto 1$ in $\freeTop$, its interpretation $\HOLinterp[e]$ is a cover $(\HOLinterp[e]_1,e):(\HOLinterp[A]_1,A,\rho_A)\to (1,1,\rho_1)$.
  \[ \begin{tikzcd}
    A \ar[d,->>,"e"'] & & \HOLinterp[A]_1 \ar[r,"\rho_A"] \ar[d,->>,"{\HOLinterp[e]_1}"'] & \Gamma(A) \ar[d,"\Gamma(e)"]
  \\ 1 & & \mathllap{1 = {}}\HOLinterp[1]_1 \ar[r,"\rho_1"] & \Gamma(1) \mathrlap{{}\cong 1}
    \end{tikzcd}
  \]
Then $\HOLinterp[e]_1$ is a cover, by \cref{prop:commas}\zcref[noname]{item:comma-regular};
  that is, $\HOLinterp[A]_1$ inhabited.
  So taking some $x \in \HOLinterp[A]_1$, $\rho_A(x)$ gives a global section of $A$ as required.
\end{proof}

\comment{Look back over treatments in \cite{lambek-scott:intro} and \cite[1.(10)]{freyd-scedrov:categories-allegories}.  They each have abstractions of this, making it slightly less elementary, but allowing them (at least Freyd–Scedrov) to extract slightly more general conclusions.  Do we want to stick with the elementary version, or bump it up a little? Or note any of those conclusions here, or consider if we might be able to generalise them along with $\RCC$ somehow?}

\section{Ranked toposes} \label{sec:ranked-toposes}

\begin{definition}
  For $r \in \N \cup \{ \infty \}$, an \defemph{$r$-ranking} on a category $\cE$ is a sequence of distinguished classes of objects $\cE_0 \subseteq \cE_1 \subseteq \cdots \subseteq \cE_i \cdots \subseteq \ob \cE$, for $0 \leq i < r$.
  %

  A functor $f : \cE \to \cF$ of $r$-ranked categories \defemph{strictly (resp.~weakly) preserves rank} if for each $x \in \cE_i$, $fx$ lies in $\cF_i$ (resp.~isomorphic to some object of $\cF_i$).
  As usual we take \defemph{preserves rank}, unqualified, to mean the weak version.
\end{definition}

We will view the classes $\cE_i$ as full subcategories of $\cE$,
and say objects have rank $ \leq i$ if they lie in $\cE_i$.
Note rank here is always cumulative — we never distinguish the precise minimal rank of objects — 
and is not assumed invariant under isomorphism.

For the remainder of the paper, $r$ will by default range over $\N \cup \{\infty\}$.

\begin{definition}
  An \defemph{$r$-ranked topos} is a regular category $\cE$ with NNO, $r$-ranked such that: 
  \begin{enumerate}
  \item each $\cE_i$ is a regular subcategory of $\cE$ and contains the NNO;
  \item for each $0 \leq i < r-1$, each object of $\cE_i$ has a power-object in $\cE$, lying in $\cE_{i+1}$.
  \end{enumerate}
  We will often write just \defemph{$r$-topos}; of course this should not be confused with the more established higher-categorical senses of \emph{$n$-topos}. 

  A functor of $r$-toposes is \defemph{(strictly) $r$-logical} if it is (strictly) regular and rank-preserving, and (strictly) preserves the NNO and the assumed power-objects.
  More rarely, we will call a functor from an $r$-topos just \defemph{logical} if it preserves the regular structure, NNO, and assumed power-objects, but is not necessarily rank-preserving; we will note clearly when we mean this!

  Following the convention of \cref{def:cats-of-cats}, we write $\strLog[r]$, $\Log[r]$, $\LOG[r]$ for the $1$- and $2$-categories of small/locally small $r$-toposes, $r$-logical functors (strict for $\strLog[r]$), and natural isomorphisms (omitted in $\strLog[r]$).
\end{definition}

As usual, we say \defemph{strict $r$-topos} to emphasise considering a small $r$-topos as an object of $\strLog[r]$.
Note however that even then, the chosen regular structure and power objects on the rank subcategories $\cE_i$ are not assumed strictly preserved by the inclusions between ranks or into $\cE$.

\begin{remark} \label{rmk:what-if-ranks-aus}
  One could strengthen the definition of $r$-toposes to assume each rank is not just regular but in fact an \emph{arithmetic universe} (\cref{def:locos-and-au}).
  For further development of ranked toposes as a logical setting in their own right, that would probably be more natural.
  In the present article, however, there would be little benefit, and it would substantially complicate the construction of internal concrete $r$-toposes via ranked universes (\cref{sec:universes}),
  We return to this point in \cref{rmk:universes-for-aus} below.
\end{remark}

\subsection{Strict $1$-categorical aspects}

It is easy to check that $r$-toposes are essentially algebraic, along similar lines to \cref{prop:logical-cats-as-eats}:

\begin{proposition} \leavevmode \label{prop:r-toposes-as-eats}
  \begin{enumerate} \setcounter{enumi}{-1}
  \item For each $r$, $r$-ranked categories and strictly rank-preserving functors are the models and homomorphisms of an essentially algebraic theory.
  \item For each $r$, strict $r$-toposes and strict $r$-logical functors are the models and homomorphisms of an essentially algebraic theory $\thT_r$.
  \item For $r \in \N$, this theory is finitely presented. 
  \item The evident forgetful functors $\strLog[0] \from \strLog[1] \from \cdots \from \strLog[\infty]$ correspond to theory extensions $\thT_r \injto \thT_s$, and so are finitary functors of locally finitely presentable categories.
  \item \label{item:inf-topos-eat-as-colimit}
  This exhibits $\strLog[\infty]$ as a (strict $2$-)limit of the tower of categories $\strLog[r]$, in both $\CAT$ and $\LFP$ (equivalently, exhibits the theory of $\infty$-ranked toposes as the colimit of the theories of $r$-toposes).
    \qed
  \end{enumerate}
\end{proposition}

\begin{definition}
  Write $\freeTop[r]$ for the \defemph{free (strict) $r$-topos}, i.e.~the (strictly) initial object of $\strLog[r]$, for each $r \in \N \cup \{\infty \}$.
\end{definition}

We will omit \enquote{strict} and call $\freeTop[r]$ just the \defemph{free $r$-topos}, justified by \cref{prop:free-r-topos-is-bi-initial} below.

\begin{proposition} \label{prop:free-inf-top-filtered-colim}
  $\freeTop[\infty] \iso \colim_{r \in \N} \freeTop[r]$, as a strict filtered colimit of categories. 
\end{proposition}

\begin{proof}
  Direct by \cref{lem:initial-object-of-limit-lfpcat}, since as noted in \cref{prop:r-toposes-as-eats}(\ref{item:inf-topos-eat-as-colimit}), $\strLog[\infty] = \lim_{r \in \N} \strLog[r]$ both in $\CAT$ and $\LFP$.
\end{proof}

\subsection{$2$-categorical aspects}

Bi-initiality of the free strict $r$-topos in $\LOG[r]$ will follow purely formally from its $1$-initiality together with a careful construction of iso-comma $r$-toposes, just as we showed for ordinary toposes in \cref{sec:free-topos-prelims}.

\begin{proposition} \label{prop:iso-comma-r-topos}
  Suppose $\cC_0 \to[F_0] \cD \from[F_1] \cC_1$ are $r$-toposes and logical functors (not necessarily rank-preserving).
  Then:
  \begin{enumerate}[(1)]
  \item $\isocomma{F_1}{F_0}$ is an $r$-topos, where $(c_1,c_0,\varphi)$ is taken to have rank $\leq i$ just if both $c_0$ and $c_1$ do;
  \item the projection functors $P_i : \isocomma{F_1}{F_0} \to \cC_i$ are strictly $r$-logical;
  \item a functor $G : \cX \to \isocomma{F_1}{F_0}$ is $r$-logical (resp.~strictly $r$-logical, rank-preserving) if its projections $P_i G : \cX \to \cC_i$ both are;
  \item for $r$-logical $G_i : \cX \to \cC_i$, natural transformations $\gamma : F_1 G_1 \to F_0 G_0$ correspond bijectively to $r$-logical functors $G : \cX \to \comma{F_1}{F_0}$ lifting $(G_1,G_0) : \cX \to \cC_1 \times \cC_0$.
\end{enumerate}
\end{proposition}

\begin{proof}
  Straightforward modification of the proofs for ordinary toposes in \cref{prop:commas,prop:comma-is-comma}.
\end{proof}

\begin{proposition}\label{prop:free-r-topos-is-bi-initial}
  For each $r$, the free strict $r$-topos $\freeTop[r]$ is moreover bi-initial in $\LOG[r]$.
\end{proposition}

\begin{proof}
  Purely formal with the aid of the iso-comma construction, exactly as in the ordinary topos case, \cref{prop:free-topos-is-bi-initial}.
\end{proof}

\begin{proposition} \label{prop:logical-functor-preserves-rank}
  Any logical functor out of the free $r$-topos $\freeTop[r]$ preserves rank (in the default, weak sense).
\end{proposition}

\begin{proof}
  This too follows formally from the iso-comma construction.
  Given $F : \freeTop[r] \to \cE$ logical, consider the iso-comma $r$-topos $\isocomma{F}{\cE}$.
  Initiality gives a section $S$ of the projection $P_1 : \isocomma{F}{\cE} \to \freeTop[r]$; the composite $P_0 S : \freeTop[r] \to \cE$ is then a strictly rank-preserving functor naturally isomorphic to $F$, so witnesses that $F$ preserves rank.
\[ \begin{tikzcd}[/tikz/baseline=(\tikzcdmatrixname-2-1.base)]
    & \isocomma{F}{\cE} \ar[d,"P_1"] \ar[r,"P_0"]
    & \cE \ar[d,"\id_\cE"] \ar[dl,phantom,"\iso"{sloped}]
    \\ \freeTop[r] \ar[r,"\id"] \ar[ur,"S",dashed] 
    & \freeTop[r] \ar[r,"F"] & \cE.
\end{tikzcd} \qedhere \]
\end{proof}

\subsection{Comparison with ordinary toposes}

Our next goal is showing that the free $\infty$-ranked topos $\freeTop[\infty]$ is equivalent to the free ordinary topos $\freeTop$.
The guiding idea is that the categorical constructors are the same: the ranks just come along for the ride.
The ranked version will contain \enquote{doppelgangers} of logical constructions taken at different ranks; but since all the assumed operations are determined by universal properties, such doppelgangers will be canonically isomorphic, so will not change the result, up to equivalence.

\begin{definition}
  Given a topos $\cE$, we can view $\cE$ as an $\infty$-ranked topos $\cE^\full$ (the \defemph{full ranking} of $\cE$), by setting $\cE^\full_i \defeq \ob \cE$ for all $i$.
  This forms the object part of an evident (strict) $2$-functor $\LOG \to \LOG[\infty]$.
\end{definition}

\begin{definition}
  Given an $\infty$-ranked topos $\cE$, its \defemph{ranked core} is the category $\toposcore{\cE}$ with objects $\ob \toposcore{\cE} \defeq \coprod_{i \in \N} \cE_i$, and maps induced by the evident projection $\ob \toposcore{\cE} \to \ob \cE$ to give a full and faithful functor $\toposcore{\cE} \to \cE$.

  It is clear that $\toposcore{\cE}$ is a topos, with a natural $\infty$-ranking given by $(\toposcore{\cE})_i \defeq \coprod_{j \leq i} \cE_j$, making $\toposcore{\cE} \to \cE$ $r$-logical.
  \comment{(In fact \emph{strictly} $r$-logical, but we never need that.)}
  Moreover, this forms the object part of an evident (strict) $2$-functor $\LOG[\infty] \to \LOG$.
\end{definition}

\begin{proposition} \label{prop:inftopos-is-topos}
  For any $\infty$-ranked topos $\cE$, if $\toposcore{\cE} \to \cE$ is essentially surjective, then $\cE$ is a topos. \qed
\end{proposition}

\begin{proposition} \label{prop:free-inftopos-is-topos}
  The free $\infty$-ranked topos $\freeTop[\infty]$ is a topos.
\end{proposition}

\begin{proof}
  By \cref{prop:inftopos-is-topos}, as initiality gives a section of $\toposcore{(\freeTop[\infty])} \to \freeTop[\infty]$.
\end{proof}

\begin{proposition} \label{prop:free-ranked-topos-equiv-free-topos}
  The free $\infty$-ranked topos is equivalent to the free topos.
\end{proposition}

\begin{proof}
  Consider the $2$-category of toposes equipped with an $\infty$-ranking; logical (not necessarily rank-preserving) functors; and natural isomorphisms.
  $\freeTop$, with its full ranking, is certainly bi-initial there;
  but so is $\freeTop[\infty]$, since it is a topos (\cref{prop:free-inftopos-is-topos}) and logical functors out of it are automatically rank-preserving (\cref{prop:logical-functor-preserves-rank}).
  The proposition follows.
\end{proof}

\subsection{Gluing}

The standard results on Artin/Freyd gluing of toposes given in \cref{sec:cat-thy-prelims,sec:free-topos-prelims} adapt directly to ranked toposes.

\begin{proposition}[Artin gluing] \label{prop:artin-gluing-r-topos}
  Let $F : \cC \to \cD$ be a cartesian, rank-preserving functor between $r$-toposes. 
  Then equipping the comma category $\comma{\cD}{F}$ with the \defemph{levelwise} ranking (an object $p : D \to FC$ has rank $\leq i$ just if $C$, $D$ both do), it carries an $r$-topos structure making the second projection functor $P_0 : \comma{\cD}{F} \to \cC$ strictly $r$-logical, the first $P_1 : \comma{\cD}{F} \to \cD$ strictly regular, and both strictly rank-preserving.

  Moreover, this is functorial as in \cref{prop:artin-gluing-topos}: pseudo-commuting squares $\varphi : JF' \iso FI$ with $I$, $J$ $r$-logical induce $r$-logical functors $\comma{J}{\varphi} : \comma{\cD'}{F'} \to \comma{\cD}{F}$, respecting identities and composition.
\end{proposition}

\begin{proof}
  The regular structure (total and rank-wise) is furnished by \cref{prop:commas}\zcref[noname]{item:comma-regular};
  similarly, the NNO and power-objects are constructed just as in the topos case (\cref{prop:artin-gluing-topos}), noting that as $F$ preserves rank, the pullback there exists and has rank $i+1$ if the objects $C$, $D$ have rank $i$.
  Functoriality once again is lengthy but routine.
\end{proof}

\begin{longcomment}
\begin{remark}
  This is one of the places where the crude definition of $r$-toposes (rank just for objects, not maps) weakens a result unnecessarily: with a good map-based definition of rank, I think this shouldn’t require $F$ to be rank-preserving.  Probably should consolidate remarks about this point somewhere.
\end{remark}
\end{longcomment}

\begin{proposition}[Freyd gluing] \label{prop:freyd-gluing-r-topos}
  The terminal object of ${\freeTop[r]}$ is projective.
\end{proposition}

\begin{proof}
  Almost word-for-word as in \cref{prop:freyd-gluing-topos}, endowing $\Set$ with the full ranking in order to make the global sections functor rank-preserving and apply \cref{prop:artin-gluing-r-topos}.
\end{proof}

\section{Indexed and internal \texorpdfstring{$r$}{r}-toposes} \label{sec:indexed-ranked-toposes}

\subsection{Background: Indexed category theory}

The endgame of the main result makes essential use of \emph{internal} $r$-toposes in the free topos.
To cleanly pass between external and internal categories, we view them both in the more general setting of \emph{indexed} categories.
A good introduction to the theory of indexed categories can be found in \cite[B1.2--4, B2.3]{johnstone:elephant};
for now, we briefly recall just the material we need:

\begin{definition}
  Given a category $\cE$ (the \defemph{base}), an \defemph{$\cE$-indexed category} $\indC$ is a pseudo-functor $\indC :\cE^\op \to \CAT$; that is,
  for each $X \in \cE$ a category $\indC_X$ (the \defemph{fibre} category over $X$), and for each $f : Y \to X$ a \defemph{reindexing} functor $f^* : \indC_X \to \indC_Y$, all functorial in $\cE$ up to coherent isomorphism.
  
  An $\cE$-indexed functor $F : \indC \to \indD$ consists of functors $F_X : \indC_X \to \indD_X$, commuting with reindexing up to coherent isomorphism.
  With a suitable notion of natural transformations, $\cE$-indexed categories form a $2$-category which we denote $\CAT[\cE]$.
\end{definition}

\begin{remark}
  Indexed categories can alternatively be presented as \defemph{fibred categories}, essentially a reorganisation of the same data, whose theory is in some respects technically cleaner.
  Everything we do here with indexed categories can equivalently be read in fibred terms. 
  A useful comparison of the two approaches is given in \cite[Discussion~1.10.4]{jacobs:categorical-logic}.
  \comment{Previously we also recommended \cite{streicher:fibered-cats} as a general intro; but feels like that Jacobs discussion is really more relevant here, and it leads the reader to Jacobs Ch.~1, which is itself a good intro to fibred cats.}
\end{remark}

\begin{examples} \leavevmode
  \begin{enumerate}
  \item For any $\cE$, $\cC$, the \defemph{constant $\cE$-indexing} of $\cC$ has all fibres equal to $\cC$, and trivial reindexing.
  \item For any $\cC$, the \emph{standard $\Set$-indexing} of $\cC$ has fibres $\cC^X$ for $X \in \Set$, with reindexing given by precomposition.
  \item For cartesian $\cE$, the \defemph{self-indexing} of $\cE$ has fibres $\cE/X$, with reindexing given by pullback.
  When working in indexed categories over a base $\cE$, we will generally identify $\cE$ with its self-indexing.
  \item Given $\indC$ indexed over $\cE$, and $X \in \cE$, the \defemph{restriction} of $\indC$ to $\cE/X$ is the $\cE/X$-indexed category $\restrictindexed{\indC}{X}$ given by $(\restrictindexed{\indC}{X})_{(f : Y \to X)} \defeq \indC_Y$, with evident reindexing.
  The restriction of the self-indexing of $\cE$ to $\cE/X$ is precisely the self-indexing of the slice $\cE/X$.
  \end{enumerate}
\end{examples}

\begin{definition}
  An indexed category $\indC$ over a base $\cE$ is regular (resp.~a topos, etc.)\ if all its fibre categories $\indC_X$ are regular (resp.~toposes), and for each $f : Y \to X$ in $\cE$ the reindexing functor $f^*:\indC_X \to \indC_Y$  is regular (resp.~logical).
  An indexed functor $F : \indC \to \indD$ is regular (logical, etc.)\ just if its fibre functors $F_X$ all are.

  \comment{Currently never used, reinstate if wanted: We write $\REG[\cE]$, $\LOG[][\cE]$, and so on for the resulting $2$-categories of indexed categories over $\cE$.}
\end{definition}

For the next few items, fix an indexed category $\indC$ over a base $\cE$.

\begin{definition}[{\cite[B1.3.12]{johnstone:elephant}, \cite[9.5.1--4]{jacobs:categorical-logic}}] \label{def:locally-small}
  $\indC$ is called \defemph{locally $\cE$-small} if for all objects $A, B \in \indC_X$, the presheaf $(\cE/X)^\op \to \Set$ sending $f : Y \to X$ to $\indC_Y(f^*A,f^*B)$ is representable.
  That is, there is an object $\chi : \indhom[\cE]{\indC}{X}{A}{B}\to X$ in $\cE/X$ (their \defemph{hom-object}), with for each $f : Y \to X$ a bijection $\indC_Y(f^* A,f^*B) \iso \cE/X(Y,\indhom[\cE]{\indC}{X}{A}{B})$, naturally in $f$.
  
  The \defemph{universal map from $A$ to $B$} is the map $u : \chi^*A \to \chi^*B$ in $\indC_{\indhom[\cE]{\indC}{X}{A}{B}}$ corresponding to $\id_{{\indhom[\cE]{\indC}{X}{A}{B}}}$. 
\end{definition}

\begin{proposition} \label{prop:indexed-hom-functor} \leavevmode
  \begin{enumerate}
  \item When $\indC$ is locally small, its hom-objects assemble automatically into an indexed functor $\indC^\op \times \indC \to \cE$.
    In particular, for each $A \in \indC_X$, there is an $\cE/X$-indexed functor $\indyon{A} : \restrictindexed{\indC}{X} \to \cE/X$, sending each $f : Y \to X$, $B \in \indC_Y$ to $\indhom[\cE]{\indC}{Y}{f^*A}{B}$.

  \item If $F : \indC \to \indD$ is an indexed functor, with $\indC$, $\indD$ locally small, then $F$ acts on hom-objects, inducing maps $F_{A,B} : \indhom[\cE]{\indC}{X}{A}{B} \to \indhom[\cE]{\indD}{X}{FA}{FB}$, suitably naturally in $X$, $A$, $B$.
  \end{enumerate}
\end{proposition}

\begin{proof}
  By the Yoneda lemma, it suffices to give natural transformations between the presheaves that the objects $\indhom[\cE]{\indC}{X}{A}{B}$ were defined to represent.
  But this is direct in each case by unwinding the definitions. 
\end{proof}

\begin{definition}
  A map $e \in \indC_X$ is an \defemph{indexed cover} if for every $f : Y \to X$, the reindexing $f^*e$ is a cover in $\indC_Y$.
\end{definition}

\begin{proposition}
  If $\indC$ is (indexed-)regular, then a map $e \in \indC_X$ is an indexed cover in $\indC$ just if it is a cover in $\indC_X$.
\end{proposition}

\begin{proof}
  Indexed-regularity says reindexing is regular, so preserves covers.
\end{proof}

For the next few items, assume additionally that $\cE$ is regular.

\begin{definition}
  An object $P \in \indC_X$ is \defemph{indexed (cover-)projective} if for each reindexing $f : Y \to X$, map $k : f^*P \to A$, and indexed cover $e : B \coverto A$ in $\indC_Y$, there is some cover $g : Z \coverto Y$ and a lift $\hat{k}$ of $g^*k$ along $g^*e$:
  \[ \begin{tikzcd}[row sep=small]  
 & B \ar[dd,cover,"e"] & & & & g^*B \ar[dd,cover,"g^*e"] \\
 & & { } \ar[r,squiggly,maps to] & { } \\
f^*P \ar[r,"k"] & A & & &g^*f^* P \ar[r,"g^*k"] \ar[uur,"\widehat{k}",dashed] & g^*A
\end{tikzcd} \]
\end{definition}

\comment{Literature search on indexed-projectivity.  Don’t remember seeing this definition, but it surely shouldn’t be new?}

\begin{examples} \leavevmode
  \begin{enumerate}
  \item In the standard $\Set$-indexing of an ordinary regular category $\cC$, an object $A \in \cC^X$ is indexed-projective just if each $A_x$ is projective in $\cC$.

  \item In the self-indexing of a topos $\cE$, an object $A \in \cE \equiv \cE/1=\cE_1$ is indexed-projective just if it is \emph{internally projective} in $\cE$, in the topos-theoretic sense of \cite[D4.5.1]{johnstone:elephant}.
  \end{enumerate}
\end{examples}

\begin{proof}
  The $\Set$-indexing example is straightforward.
  The correspondence with internal projectivity in toposes is essentially \cite[Thm~3.3]{nlab:internally-projective-object-r12}.
  \comment{Can we find this proof in print somewhere instead? Is it in the Shulman stack semantics paper?}
\end{proof}

\begin{proposition} \label{prop:indexed-proj-equivalent-forms}
  The following are equivalent, for $P \in \indC_X$:
  \begin{enumerate}[(1)]
  \item \label{item:ind-proj} $P$ is indexed-projective.
  
  \item \label{item:ind-proj-covers} (When $\indC$ is regular.) Any cover of any reindexing of $P$ splits on some cover of its base.
    That is, for each $f : Y\to X$ and cover $e : B \coverto f^*P$, there is some cover $g : Z \coverto Y$ such that $g^*e$ splits in $\indC_Z$. 
  
  \item \label{item:ind-proj-representable} (When $\indC$ is locally small.) The indexed representable functor $\indyon{P} : \restrictindexed{\indC}{X} \to \cE/X$ preserves covers.
  \qed
  \end{enumerate}
\end{proposition}

\begin{proof}
  \zcref[noname]{item:ind-proj} $\Iff$ \zcref[noname]{item:ind-proj-covers} adapts easily from the ordinary case, \cref{prop:projective-equivalent-split}.

  \zcref[noname]{item:ind-proj-representable} $\Imp$ \zcref[noname]{item:ind-proj}:
  Suppose $\indyon{P}$ preserves covers.
  Then for any reindexing $f : Y \to X$, and cover $e : B \coverto A$ and map $k : f^* P \to A$ in $\indC_Y$, a cover of $Y$ on which $k$ lifts to $B$ is given by the following pullback:
  \[ \begin{tikzcd}
  \bullet \ar[r,dashed,"\widehat{k}"] \ar[d,dashed,cover] \ar[dr,drpb]
  & \indhom[\cE]{\indC}{Y}{f^*P}{B} \ar[d,cover,"(\indyon{P})_f(e)"{description,pos=0.45}]
  \\ Y \ar[r,"k"]
  & \indhom[\cE]{\indC}{Y}{f^*P}{A}
  \end{tikzcd} \]
  
  \zcref[noname]{item:ind-proj} $\Imp$ \zcref[noname]{item:ind-proj-representable}:
  Suppose $P$ is indexed-projective.
  Given $f : Y \to X$ and $e : B \coverto A$ in $\indC_Y$,
  consider the universal map from $f^*P$ to $A$ in $\indC_{\indhom[\cE]{\indC}{Y}{f^*P}{A}}$, denoted $u : \chi^*f^*P \to \chi^*A$ as in \cref{def:locally-small}.
  We know $u$ must lift to a map $\widehat{u} : g^*\chi^*f^*X \to g^*\chi^*B$ on some cover $g : Z \coverto \indhom[\cE]{\indC}{Y}{f^*P}{A}$;
  but that amounts exactly to a triangle of the following form
  \[ \begin{tikzcd}
  & \indhom[\cE]{\indC}{Y}{f^*P}{B} \ar[d,"(\indyon{P})_f(e)" description]
  \\ Z \ar[r,cover,"g",dashed] \ar[ur,"\widehat{u}",dashed]
  & \indhom[\cE]{\indC}{Y}{f^*P}{A},
  \end{tikzcd} \]
  implying by right-cancellation that $(\indyon{P})_f(e)$ is a cover as required.
  \comment{Is this clear enough without including the maps to $Y$ in the diagrams, now that the universal map is presented explicitly back in \cref{def:locally-small}?}
\end{proof}

\begin{definition}[{\cite[B2.3.1--2.3.3]{johnstone:elephant}, cf.\ \cite[\textsection 7.3]{jacobs:categorical-logic}}]
  An internal category $\smC$ in $\cE$ can be viewed as an $\cE$-indexed category $\extcat{\smC}$, its \defemph{externalisation}, with $\ob \extcat{\smC}_X \defeq \cE(X,\ob \smC)$, and so on.

  An $\cE$-indexed category is called \defemph{$\cE$-small} if it is equivalent to the $\cE$-externalis\-ation of some internal category.
\end{definition}

\begin{proposition}
  An $\cE$-small indexed category is locally $\cE$-small.
\end{proposition}

\begin{proof}
  It suffices to check for the externalisation of an internal category, which is direct:
  for $A,B \in \ob \extcat{\smC}_X = \cE(X,\ob \smC)$, their hom-object is the pullback of $\mor \smC \to (\ob \smC)^2$ along $(A,B) : X \to (\ob \smC)^2$.
\end{proof}

We will generally identify internal categories with their externalisations, justified by the following proposition: 

\begin{proposition}[{\cite[B2.3.3]{johnstone:elephant}, cf.\ \cite[7.3.8]{jacobs:categorical-logic}}] \label{prop:externalisation-2-fully-faithful}
  The externalisation of internal categories forms a $2$-functor $\intCat[\cE] \to \CAT[\cE]$, which is $2$-fully-faithful, in that the maps on hom-categories $\intCat[\cE](\smC,\smD) \to \CAT[\cE](\extcat{\smC},\extcat{\smD})$ are equivalences. \qed
\end{proposition}

\begin{proposition} \label{prop:internal-regular-category}
  An internal category is regular (resp.~a topos), viewed as an indexed category, just if it admits the structure of an internal regular category (topos).
\end{proposition}

\begin{proof}
  Any model $M$ in $\cE$ of any essentially algebraic theory $\thT$ induces by Yoneda a functor $\cE^\op \to \Mod{\thT}$;
  in particular, an internal regular category $\smC$ yields a functor $\cE^\op \to \strReg$, whose underlying indexed category is just the externalisation of $\smC$.

  Conversely, if an internal category $\smC$ is indexed-regular, then regular structure on $\smC$ can be chosen as the universal instances of the corresponding operations on the externalisation.
  For instance, the product operation $\ob\smC \times  \ob\smC \to  \ob\smC$ arises as the product of the pair $\pi_0$, $\pi_1$ in $\extcat{\smC}_{(\ob \smC \times  \ob\smC)}$.
  
  The topos case is directly analogous.
\end{proof}

\subsection{Indexed and internal $r$-toposes}

\begin{definition}
  An \defemph{indexed $r$-topos} over $\cE$ is an $\cE$-indexed category $\indC$ together with $r$-topos structure on each fibre, and such that all reindexing functors are $r$-logical.

  Given an indexed $r$-topos $\indC$ over $\cE$, the indexed subcategory $\indC_i$ is defined by taking $(\indC_i)_X \defeq (\indC_X)_i \subseteq \indC_X$ as expected; for the reindexing functor along $f : Y \to X$, we know that $f^* : \indC_X \to \indC_{Y}$ preserves rank (albeit weakly), so we may take some perturbation of it up to isomorphism which maps $(\indC_X)_i$ strictly into $(\indC_Y)_i$.
  There are evident full and faithful indexed functors $\indC_i \to \indC$ and $\indC_j \to \indC_i$ as in the non-indexed case.
\end{definition}

Most results and constructions of \cref{sec:ranked-toposes} now adapt more or less straightforwardly to indexed and internal versions, compatibly with externalisation.

\begin{proposition} \label{prop:free-r-topos-indexed-initial}
  For any category $\cE$, $\freeTop[r]$ with the \emph{constant} $\cE$-indexing is (bi-)initial among $\cE$-indexed $r$-toposes and $r$-logical functors. \qed
\end{proposition}

\begin{proposition}[Indexed Artin gluing] \label{prop:indexed-artin-gluing}
  Let $F : \indC \to \indD$ be a cartesian, rank-preserving functor of indexed $r$-toposes over $\cE$.
  Then the indexed comma category $\comma{\indD}{F}$ carries an $r$-topos structure, given fibrewise by the levelwise ranking as in \cref{prop:artin-gluing-r-topos}.

  Moreover, an $r$-logical functor $H : \indD \to \indD'$ induces an $r$-logical functor $\comma{\indD}{F} \to \comma{\indD'}{GF}$, strictly over $\indC$. \qed
\end{proposition}

\begin{proposition}[Internal Artin gluing] \label{prop:internal-artin-gluing}
  Let $F : \smC \to \smD$ be a cartesian, rank-preserving functor of internal $r$-toposes, in a cartesian category $\cE$.
  Then the internal comma category $\comma{\smD}{F}$ carries an internal $r$-topos structure, whose externalisation is precisely the indexed Artin gluing of \cref{prop:indexed-artin-gluing};
  moreover, the first projection $P_0 : \comma{\smD}{F} \to \smC$ is a strict map of internal $r$-toposes.
\end{proposition}

\begin{proof}
  This amounts to checking that the $r$-topos structure on the comma category in \cref{prop:artin-gluing-r-topos} is all essentially-algebraically defined.
\end{proof}

\begin{proposition} \label{prop:internal-r-topos-from-indexed} \leavevmode
  \begin{enumerate}[(1)]
  \item \label{item:indexed-r-topos-from-internal} The externalisation of an internal ranked category (resp.~$r$-topos) carries an indexed ranking (resp.~$r$-topos structure).
  \item \label{item:internal-r-topos-from-ranked} If $\smC$ is an internal ranked category in $\cE$ whose externalisation is an $r$-topos, then $\smC$ carries the structure of an internal $r$-topos.
  \end{enumerate}
\end{proposition}

\begin{proof}
  Part \zcref[noname]{item:indexed-r-topos-from-internal} is immediate by Yoneda, as in \cref{prop:internal-regular-category}.

  For part \zcref[noname]{item:internal-r-topos-from-ranked}, regular structure on $\smC$ and the subcategories $\smC_{\leq i}$ follows by \cref{prop:internal-regular-category}, along with regularity of the subcategory inclusions;
  and the maps $P_i : \ob \smC_i \to \ob \smC_{i+1}$ providing power-object structure may be constructed along the same lines as the proof of \cref{prop:internal-regular-category}.
\end{proof}

\comment{Possibly note somewhere comparison between indexed functors to the base category and “internal presheaves” as defined in e.g.\ Johnstone, but how the indexed approach makes the structure-preservation of the functor clearer? Maybe unnecessary.}

\subsection{Internal free $r$-toposes in arithmetic universes} \label{sec:internal-free-models}

We now turn to the question of \emph{internal} free $r$-toposes: first their existence, and then a closer inspection to determine that (as expected) they can be constructed syntactically, and hence admit a form of Gödel-numbering.

An appropriate minimal categorical setting for syntactic constructions is provided by \emph{locoses} and \emph{arithmetic universes}, introduced by Joyal and developed by Maietti and others \cite{joyal-wraith-1979,maietti:aus-via-tt}.
We briefly recall their key points here; for a more thorough introduction, see \cite{maietti:joyals-aus}. 

\begin{definition} \label{def:locos-and-au}
  A \defemph{locos} is a lextensive category with parametrised list objects.
  
  An \defemph{arithmetic universe} (AU), or \defemph{list-arithmetic pretopos}, is a locos that is moreover exact; equivalently, a pretopos with parametrised list objects.
  
  A \defemph{(locos-/AU-)logical functor} is a lextensive (resp.~pretopos) functor preserving list objects.

  We write $\LOCOS$, $\AU$ for the $2$-categories of locoses/AU’s, suitably logical functors, and natural transformations.  
\end{definition}

Joyal originally introduced arithmetic universes to give a categorical account of Gödel’s incompleteness theorems \cite{joyal-wraith-1979}.
A key step is showing the structure of an AU suffices to construct a free \emph{internal} AU
— a categorical analogue of arithmetic’s ability to encode its own syntax.
Analysing what that construction relies on leads to the generalisation:

\begin{theorem}[{\cite[Thm.~3.17]{maietti:lex-sketches-in-au}}] \label{thm:initial-models-in-aus}
  Any finitely presented essentially algebraic theory has an initial internal model in any AU, and AU-logical functors preserve these.
\end{theorem}

Unfortunately, while the result is well-known in folklore, the only full presentation we are aware of is in the unpublished manuscript \cite{maietti:lex-sketches-in-au}; the result has never appeared in the published literature, as discussed in \cite{lumsdaine-2019:mo-free-models-of-fp-eats}.
The nearest published results we are aware of are:
\begin{enumerate}
  \item \cite{palmgren-vickers} gives a clean presentation of a syntax for essentially algebraic theories, sufficiently elementarily that (with some work) the whole presentation can be internalised in an arithmetic universe, via e.g.~the type theory for them developed in \cite{maietti:joyals-aus}.
  \item \cite[Ch.~7]{morrison:masters} constructs the initial internal AU in an AU, and states the general theorem, but leaves part of the proof conjectural.
  \item \cite{maietti:fp-sketches-in-au} gives the result for EATs presented by finite-\emph{product} sketches.
  \item \cite[\textsection 7]{maietti:joyals-aus} constructs several more initial internal categories with extra logical structure, heuristically demonstrating all the techniques needed for the general theorem.
\end{enumerate}

The specific case we require is just:

\begin{theorem} \label{thm:internal-free-r-topos-exists-and-preserved}
  For each $r \in \N$, any arithmetic universe $\cE$ has an initial internal strict $r$-topos, which we denote $\freeTop[r][\cE]$; and AU-logical functors preserve these.
\end{theorem}

\begin{proof}
  By \cref{thm:initial-models-in-aus}; or more concretely, by adapting the constructions of the initial internal AU from \cite[Ch.~7]{morrison:masters}, \cite[\textsection 7]{maietti:joyals-aus}.
\end{proof}

We will make extensive use of these internal free $r$-toposes.
However, for our endgame, we need to squeeze a bit more juice out of their construction.
As described in the sketch, we will need that elements of the free $r$-topos admit Gödel-numbering — that is, are internally enumerable — and hence admit choice functions by taking minimal witnesses.

We therefore develop the basic theory of \emph{enumerable} objects in locoses and arithmetic universes --- their choice properties, and their appearance in initial algebraic structures.
Most ideas involved appear in earlier work \cite{joyal-wraith-1979,morrison:masters,maietti:joyals-aus}, but not quite in the forms we require.

\begin{definition}
  A \defemph{predicate} in a locos%
  \footnote{Predicates are usually defined assuming just finite products and an NNO, using the idempotent $\min (1,-) : \NNO \to \NNO$ to simulate $2 = \{0,1\}$ as a formal retract of $\NNO$.}
  is a map $\NNO \to 2$.
  The \defemph{realisation} of a predicate $p$ is the pullback along $p$ of $1 \to 2$, a complemented subobject of $\NNO$.
  Write $\Pred(\cE)$ for the category of predicates in $\cE$ and maps in $\cE$ between their realisations.

  Call an object of $\cE$ \defemph{strictly enumerable} if it admits a complemented monomorphism to $\NNO$, or equivalently lies in the essential image of $\Pred(\cE)$.
  More generally, call a map $p : B \to A$ in $\cE$ strictly enumerable if it is so as an object of the slice $\cE/A$.

  When $\cE$ is regular, call an object of $\cE$ \defemph{enumerable} if it is covered by (i.e.~admits a cover from) a strictly enumerable object; call a map enumerable if it is enumerable as an object of the slice.
\end{definition}

\begin{proposition}
  Strictly enumerable maps in any locos (resp.~enumerable maps in a regular locos) are stable under pullback, closed under composition, and preserved by (regular) locos-logical functors.
\end{proposition}

\begin{proof}
  Direct, using the standard encoding $\NNO^2 \injto \NNO$ for composition.
\end{proof}

\begin{proposition} \label{prop:enum-to-enum-diag-enum}
  In a regular locos, any map from an enumerable object to an object with enumerable diagonal is enumerable.
  In particular, any map from an enumerable to a strictly enumerable object is enumerable.
\end{proposition}

\begin{proof}
  For any $f : B \to A$, the graph factorisation $(\id,f) : B \to B \times A \to A$ exhibits $f$ as a pullback of the diagonal $A \to A \times A$ followed by a pullback of $B \to 1$; so if these are both enumerable, so is $f$.
  
  For the strictly enumerable case, note that any subobject of $\NNO$ inherits its decidable equality, so certainly has enumerable diagonal.
\end{proof}

That proposition, together with the next few results, amount to an analysis and generalisation of the construction of images in $\Pred \cE$ from \cite[Proposition 4.7]{maietti:joyals-aus}.

\begin{definition}
  Say a map has \emph{split image} if it factors as a split epi followed by a mono.
\end{definition}

\begin{proposition} \label{prop:split-image-facts}
  In any cartesian category,
  \begin{enumerate}[(1)]
  \item a split image is a stable image factorisation;
  \item \label{item:split-image-cover-splits} any cover with split image is split;
  \item \label{item:split-image-if-covered} if $f : B \to A$ is covered by $g : \bar{B} \to A$ (i.e.~$g = ef$, with $e$ a cover), and $g$ has split image, then so does $f$. \qed
  \end{enumerate}
\end{proposition}

\begin{proposition} \label{prop:enum-map-split-image}
  Any enumerable map has split image.
\end{proposition}

\begin{proof}
  Given a strictly enumerable object $A = \AUinterp[ x \of \NNO \suchthat \alpha(x) = 1]$, a split image for $A \to 1$ is given by the subobject of \emph{minimal} representatives, $\AUinterp[ x \of \NNO \suchthat \alpha(x) = 1 \land \bigland_{y < x} (\alpha(y) = 0) ]$.

  The case for a general strictly enumerable map follows by passing to the slice over its base; and for an enumerable map, by  \cref{prop:split-image-facts}\zcref[noname]{item:split-image-if-covered}.
\end{proof}

\begin{proposition} \label{prop:enum-to-strenum-split-image}
  Any map from an enumerable object to an object with enumerable diagonal (in particular, any strictly enumerable object) has split image. Any cover between such objects splits.
\end{proposition}

\begin{proof}
  By \cref{prop:enum-map-split-image} together with \cref{prop:enum-to-enum-diag-enum,prop:split-image-facts}\zcref[noname]{item:split-image-cover-splits}.
\end{proof}

\begin{proposition}[{\cite[Proposition 4.7]{maietti:joyals-aus}}] \label{prop:strenum-is-reg-sublocos}
  Let $\cE$ be a locos.
  Then $\Pred{\cE}$ is a regular sub-locos of $\cE$, in which every cover splits; equivalently, all objects are projective.
\end{proposition}

\begin{proof}
  Closure under the locos structure of $\cE$ is generally straightforward;
  regularity and splitting are by \cref{prop:enum-to-strenum-split-image}.
\end{proof}

Our next main goal is the fact that in the internal free $r$-topos $\freeTop[r][\cE]$, each type is enumerable.
Classical Gödel-numbering achieves this by constructing the free model entirely within the world of formal quotients of recursive subsets of $\N$.
Categorically, this amounts to lifting $\freeTop[r][\cE]$ to a suitable category of such formal quotients: the exact completion of $\Pred{\cE}$.

We recall the definition briefly; for full details see \cite{carboni-vitale:reg-and-ex}, \cite[A3.3]{johnstone:elephant}.

\begin{definition}
  The \defemph{exact-over-lex} or briefly \defemph{ex/lex completion} of a cartesian category $\cE$, denoted $\exlex{\cE}$, has objects pseudo-equivalence relations in $\cE$, i.e.~maps $R \to A \times A$ in $\cE$ satisfying reflexivity, symmetry, and transitivity.
  The \defemph{exact-over-regular} or \defemph{ex/reg completion}, $\exreg{\cE}$, is the full subcategory of $\exlex{\cE}$ on equivalence relations, i.e.~objects with $R \injto A \times A$ mono.
\end{definition}

\begin{proposition} \label{prop:ex-completions-universal}
  If $\cE$ is cartesian (resp.~regular), then $\exlex{\cE}$ (resp.~$\exreg{\cE}$) is exact,
  and these constructions form left bi-adjoints to the forgetful $2$-functors $\EX \to \LEX$, $\EX \to \REG$.

  If $\cE$ is additionally lextensive, then so are $\exlex{\cE}$ and $\exreg{\cE}$, and these form left adjoints to the forgetful $2$-functors $\PRETOP \to \LEXT$, $\PRETOP \to \REGLEXT$.
\end{proposition}

\begin{proof}
  The universal properties in $\EX$ are well known, given in for instance \cite[Thm.~29]{carboni-vitale:reg-and-ex}, \cite[A3.3.10, A3.3.13]{johnstone:elephant}.
  For the lextensive/pretopos case, we are unaware of a source stating the universal properties, but they follow straightforwardly from the construction of coproducts/lextensivity in \cite[Lemma~2.2]{carboni:some-free-constructions}, \cite[Thm.~4.1.1]{menni:phd-thesis}.
\end{proof}

\begin{theorem} \label{thm:ex-of-locos-is-au}
  The ex/lex (resp.~ex/reg) completion of a locos (resp.~regular locos) is an arithmetic universe.
  Moreover, these form left bi-adjoints to the forgetful $2$-functors $\AU \to \LOCOS$, $\AU \to \REGLOC$.
\end{theorem}

\begin{proof}
  The special case on the locos of predicates of a Skolem theory goes back to \cite[\textsection 2]{joyal-wraith-1979}, and is presented thoroughly in \cite[Prop.~4.10]{maietti:joyals-aus} and analysed further in \cite{maietti-trotta:quotients-etc}. 
  
  Building on \cref{prop:ex-completions-universal}, we just need to show (for the ex/lex case) that if $\cE$ is a locos, then
  $\exlex{\cE}$ has list objects, the unit $\cE \to \exlex{\cE}$ preserves them, and for any locos functor from $\cE$ to an AU $\cF$, its extension to an exact functor $\exlex{\cE} \to \cF$ also preserves them;
  and similarly for the ex/reg case.
 
  Given a pseudo-equivalence relation $R \parpair{}{} A$, $\List R \parpair{}{} \List A$ is also a pseudo-equivalence relation (since $\List$ commutes with pullbacks),
  giving a parametrised list object for $R \parpair{}{} A$ in $\exlex{\cE}$,
  and which moreover is mono (i.e.~lies in $\exreg{\cE}$) in case $R$ was.
  
  The unit functor from $\cE$ clearly preserves list objects.

  Finally, for a locos map $F : \cE \to \cF$ with $\cF$ an AU, checking that the extension $\bar{F} : \exlex{\cE} \to \cF$ preserves these (and $\exreg{\cE} \to \cF$ in case $\cE$, $F$ regular) amounts to checking that $\List$ commutes with quotients of equivalence relations in $\cF$, which is lengthy but routine.
  \comment{Is “lists commute with quotients” in the AU literature??}
  (The special case of natural numbers objects is given in \cite[Rmk.~5.4]{emmenegger-palmgren:exact-completion}, adapting \cite[Cor.~5.2]{birkedal-carboni-rosolini-scott}.)
\end{proof}

In particular, for any arithmetic universe $\cE$, $\exlex{(\Pred \cE)}$ is also an AU, with an AU-logical \defemph{realisation} functor $\exlex{(\Pred \cE)} \to \cE$ extending the realisation of predicates.
(Note that this is \emph{not} generally full or faithful.)
Objects of $\exlex{(\Pred \cE)}$, or their realisations, may be called \defemph{enumerably presented} objects of $\cE$; in particular, the realisations are clearly enumerable.

With hindsight, the locos/AU structures of $\Pred(\cE)$ and $\exlex{(\Pred \cE)}$ are clearly discernible in classical proof theory:
they abstract the constructions that Gödel-numbering uses to build syntax and free algebras entirely within the world of recursive sets and their formal quotients.
Concretely, key consequences of Gödel-numbering now follow directly for AU’s:

\begin{corollary} \label{cor:au-godel-numbering}
  In any AU, each type of the internal free $r$-topos (or more generally, the initial model of any finitely presented EAT) is enumerable.
\end{corollary}

\begin{proof}
  Since realisation $\exlex{(\Pred \cE)} \to \cE$ is AU-logical, it preserves initial models (\cref{thm:initial-models-in-aus}), and in particular sends $\freeTop[r][\exlex{(\Pred \cE)}]$ to $\freeTop[r][\cE]$.
  %
\end{proof}

We could equivalently have used $\exreg{(\Pred \cE)}$ here: the two are equivalent since in $\Pred \cE$ all objects are projective, so by \cite[Prop.~9(iii)]{carboni-vitale:reg-and-ex}, $\reglex{(\Pred \cE)} \equiv \cE$ and $\exlex{(\Pred \cE)} \equiv \exreg{(\Pred \cE)}$.

\begin{caveat}
  It is very tempting to think that realisations of objects of $\exlex{(\Pred \cE)}$ have enumerable diagonal, and hence that covers between them should split according to \cref{prop:enum-to-strenum-split-image}.

  Indeed, given a pseudo-equivalence relation $R \parpair{}{} A$ in $\Pred \cE$, we know the pullback of the diagonal $\Delta_{A/R}$ to $A \times A$ along the canonical cover $A \coverto A/R$ is precisely $R \to A \times A$ (by exactness), and hence enumerable by \cref{prop:enum-to-enum-diag-enum}.
  However, this does \emph{not} imply that the diagonal $\Delta_{A/R}$ is itself enumerable: enumerability does not in general descend along covers.
\end{caveat}

\section{Reflection results and applications} \label{sec:reflection-results}

We can now give the main reflection result: the standard interpretation of $\IHOL[r]$ internally to a topos, organised via \emph{$r$-universes}, sequences of universes suitably closed under the constructions of $\IHOL[r]$.
Equipped with this, we will then be able to first run the Freyd gluing argument internally to a topos, and then bump it up to our main result, projectivity of $\NNO$ in the free topos.

\subsection{Universes and the standard interpretation} \label{sec:universes}

Reflection arguments rely essentially on the logic in question being able to host the \emph{standard interpretation} for its fragments.
Constructing this typically involves building a sufficiently large universe of sets or predicates to host the interpretation.

In our case, to internalise Freyd gluing for $r$-toposes within a topos $\cE$, we need the standard interpretation of its internal free $r$-topos; that is, an $r$-logical functor $\freeTop[r][\cE] \to \cE$ (of $\cE$-indexed $r$-toposes, identifying $\freeTop[r][\cE]$ with its externalisation and $\cE$ with its self-indexing as usual).
For this, it suffices to construct some \emph{internal sub-$r$-topos} of $\cE$; that is, an internal $r$-topos $\smE'$ whose externalisation is an indexed sub-$r$-topos of $\cE$.
As $\freeTop[r][\cE]$  interprets by initiality into $\smE'$, the composite $ \freeTop[r][\cE] \to \smE' \to \cE$ then yields the desired interpretation.

In this section, we construct such an internal sub-$r$-topos in any topos, by building a \defemph{$r$-ranked universe} — a sequence of universes closed under type constructions analogously to the levels of an $r$-topos.

Precisely, by a \defemph{family} or \defemph{universe} in a category with pullbacks, we simply mean a map, viewed as a family of objects indexed by its base; we will typically say “universe” just when the family is closed under some logical constructions, but will always state such closure assumptions explicitly.
We typically write a family as $q : \Ut \to \Ub$, and call it by metonymy just $\Ub$ \cite{thurber:metonymy}.
\comment{Would it be clearer to make the notation more explicit with, say, $q_\Ub$ for the map?}\comment{H\textcent\textcent: No, just complicates the notation for no gain}

Note that some literature, following \cite[Def~2.1]{streicher:universes-in-toposes}, uses \emph{universe} to mean a class of maps --- essentially what we present as the \emph{indexed subcategory} associated to a universe.
In their terms, our universes are the \emph{(weakly) generic maps} for their universes.
\comment{Possibly add slightly longer pointer to some such literature, e.g. \cite{streicher:universes-in-toposes}, \cite{gratzer-shulman-sterling:strict-univs}}

Fix for the rest of this section an ambient regular category $\cE$ with $\NNO$.
\comment{For much of it, just “cartesian” would suffice. Do we want to be more careful about what’s required where?}

\begin{definition}
  We say a family $q : \Ut \to \Ub$ is \defemph{contained in} another $q' : \Ut' \to \Ub'$ just if $q$ is a pullback of $q'$,
  and call families \defemph{equivalent} if they are contained in each other.
\end{definition}

\comment{An alternative here would be \enquote{there is some cover $e : V \to \Ub$ such that $e^*q$ is a cover of $q$}. Arguably that’s the more natural version in regular categories; but equally arguably, using the present version is just part of our general commitment to \cref{convention:chosen-structure}.}

\begin{proposition} \label{prop:universe-replace-mono}
  Suppose $\cE$ is lextensive.
  If a family $\Ub$ is contained in another $\Ub'$, there is a family $\Ub''$ equivalent to $\Ub'$ for which the containment of $\Ub$ is witnessed by a monomorphism. 
\end{proposition}

\begin{proof}
  Take $\Ut'' \to \Ub''$ to be just $\Ut + \Ut' \to \Ub + \Ub'$.
\end{proof}

\begin{definition}
  Let $q : \Ut \to \Ub$ be any family in $\cE$.
  Then we write $\indU$ for the \defemph{full $\cE$-indexed subcategory of $\cE$ on objects from $\Ub$}: that is, the $\cE$-indexed category $\ob \indU_X \coloneq \cE(X,\Ub)$, and maps induced from $\cE$ by the mapping $\indU \to \cE$ sending $f : X \to \Ub$ to $f^*\Ut$, so that the resulting $\cE$-indexed functor $\indU \to \cE$ is full and faithful (considering $\cE$ with its self-indexing throughout, as usual).
  Equivalently, one could define $\indU_X$ as the full subcategory of $\cE/X$ on objects that are contained in $\Ub$, as families.
\end{definition}

\begin{proposition} \label{prop:subuniverse-gives-subcategory}
  For families $\Ub$, $\Ub'$, the family $\Ub$ is contained in (resp.~equivalent to) $\Ub'$ just if $\indU \subseteq \indU'$ (resp.~$\equiv$) as indexed subcategories of $\cE$. \qed
\end{proposition}

\begin{definition}
  Given a family $q : \Ut \to \Ub$ say that $\Ub$ is closed under binary products (resp.~binary sums, subobjects), or contains $1$ (resp.~$\NNO$) if $\indU$, is closed under these as a subcategory of $\cE$.

  More generally, say $q' : \Ut' \to \Ub'$ has power-objects (products, etc.)~for $\Ub$ if $\indU'$ has power-objects (products, etc.)~for $\indU$ as subcategories of $\cE$.
\end{definition}

\comment{We never actually use closure under sums. Take it out, or keep it since it’s perhaps helpful for understanding the general pattern?}

Translating this to more concrete data, we see:
\begin{proposition}
  Given families $q : \Ut \to \Ub$, $q' : \Ut' \to \Ub'$ in $\cE$:
  \begin{enumerate}
  \item $\Ub$ contains $\NNO$ (resp.~$1$) just if there is some global element $e : 1 \to \Ub$ such that $e^*\Ut$ is an NNO (resp.~terminal) in $\cE$;

  \item $\Ub$ is closed under binary products (resp.~sums) just if there is $f : \Ub \times \Ub \to \Ub$ such that $f^*\Ut$ is a product (resp.~sum) for $\pi_0^*\Ut$ and $\pi_1^*\Ut$ in $\cE/\Ub \times \Ub$;

  \item $\Ub$ is closed under subobjects just if for every $f : X \to U$ and mono $m : Y \injto f^*\Ut$, the composite $\pi_0 m : Y \to X$ is a pullback of $q$; or equivalently (if $q$ has a power-object $\pow_{\Ub} \Ut$ in $\cE/\Ub$) just if there is some map $\pow_{\Ub} \Ut \to \Ub$ exhibiting the composite ${\epsilon} \injto \Ut \times_{\Ub} \pow_{\Ub} \Ut \to \pow_{\Ub}$ as a pullback of $q$;

  \item $\Ub'$ has power-objects for $\Ub$ just if there is $f : \Ub \to \Ub'$ such that $f^*\Ut'$ is a power-object for $\Ut$ in $\cE/\Ub$;
  \end{enumerate}
\end{proposition}

\begin{proof}
  In each case, the given operations on $\Ub$ directly provide categorical closure of the corresponding subcategory.
  Conversely, categorical closure yields operations on $\Ub$ as the universal instance of each closure condition.
  For instance, if $\indU$ is closed under products, then the product of the pair $\pi_0$, $\pi_1$ in $\indU_{\Ub \times \Ub}$ — the \enquote{family of all products of objects from $\Ub$} — is precisely a product operation $\Ub \times \Ub \to \Ub$ as required.
\end{proof}

\begin{definition}
  A \defemph{simple universe} in a regular category with NNO is a family closed under binary products and subobjects, and containing $1$ and $\NNO$.
\end{definition}

\begin{definition}
  An \defemph{$r$-universe} $\Ub_\bullet$ is a simple universe $q : \Ut \to \Ub$, with a sequence of subobjects $\Ub_0 \injto \Ub_1 \injto \Ub_i \injto \ldots \Ub$ (for $i \leq r$), such that considering the $\Ub_i$ as universes by pulling back $q$, each $\Ub_i$ is a simple universe, and $\Ub_{i+1}$ has power-objects for $\Ub_i$.
\end{definition}

\begin{proposition} \label{prop:universe-closed-iff-indexed-closed} \label{prop:simple-universe-reg-subcat}
  For any simple universe $\Ub$, $\indU$ is a regular indexed subcategory of $\cE$ with $\NNO$.
  For any $r$-universe $\Ub_\bullet$, the sequence of subcategories $\indU_i \subseteq \indU$ form an indexed $r$-topos, and the inclusion to $\cE$ is logical.
\end{proposition}

\begin{proof}
  Closure of $\indU$ under subobjects in $\cE$ implies closure under equalisers and images, and hence (together with $\times$ and $1$) closure under all regular structure of $\cE$.
\end{proof}

\begin{proposition}[{Cf.\ \cite[7.1.4(ii),7.3.4(ii)]{jacobs:categorical-logic}}]\label{prop:exp-univ-is-small}
  If a family $\Ut \to \Ub$ is exponentiable, then $\indU$ is $\cE$-small.
\end{proposition}

\begin{proof}
  The internal avatar $\smU$ (which goes back to \parencite[\textsection 4]{bunge:internal-presheaves}) is defined by taking $\ob \smU \defeq \Ub$, and $\mor \smU$ the exponential $(\Ub \times \Ut)^{\Ut \times \Ub}$ in $\cE/\Ub \times \Ub$.
  It is direct to check that $\smU \equiv \indU$.
\end{proof}

\begin{proposition}
  Suppose $\cE$ is a topos.
  Then any simple universe $\Ub$ in $\cE$ determines a full internal regular subcategory of $\cE$;
  and any $r$-universe $\Ub_\bullet$ determines a full internal sub-$r$-topos $\smU$ of $\cE$, with $\smU_i \equiv \indU_i$ for each $0 \leq i < r$.
  \comment{Check this notation remains clear if we change typefaces further.}
\end{proposition}

\begin{proof}
  In each case, \cref{prop:exp-univ-is-small} gives an internal avatar $\smU \equiv \indU$.

  For $\Ub$ a simple universe, 
  $\indU$ is indexed-regular by \cref{prop:simple-universe-reg-subcat}, so $\smU$ carries internal regular structure by \cref{prop:internal-regular-category}.

  For $\Ub_\bullet$ an $r$-universe, the subobjects $\Ub_i \injto \Ub = \ob \smU$ give an internal ranking on $\smU$.
  The externalisations of the rank subcategories $\smU_i$ are equivalent to the subcategories $\indU_i \subseteq \indU$, so form an indexed $r$-topos by \cref{prop:simple-universe-reg-subcat},
  and thus an internal one by \cref{prop:subuniverse-gives-subcategory}.
\end{proof}

So $r$-universes yield internal sub-$r$-toposes, as desired;
it remains to show such universes exist, in a rich enough ambient category.

\begin{proposition}
  Any family $q : \Ut \to \Ub$ in an arithmetic universe is contained in a family closed under products.
\end{proposition}

\begin{proof}
  Take the family $\List(q) : \List(\Ut) \to \List(\Ub)$.
  This contains $\Ub$ as the singleton lists, and is closed under products: the fibre of $\List(q)$ over a list from $\Ub$ is the product of the corresponding fibres of $q$, and so concatenation of lists gives a product operation for the fibres.
  Careful verification of this is lengthy but straightforward with the internal type theory, along similar lines to an ordinary proof in $\Set$. \todo{find the closest citation we can; e.g.\ the Maietti result characterising equality in list objects is fairly close}.
\end{proof}

\begin{proposition}
  Any family $q : \Ut \to \Ub$ in a topos is contained in a family $\Ub'$ closed under subobjects, and moreover closed under products if $\Ub$ was.
\end{proposition}

\begin{proof}
  Take $q'$ to be ${\in} \to \pow_{\cE/\Ub}(\Ut)$, i.e.\ the power-object of $\Ut$ in $\cE/\Ub$.
  This is always closed under subobjects, since a subobject of a subobject is a subobject; and if $\Ub$ was closed under products, then this is too, since a product of subobjects is a subobject of a product.
  Again, careful verification is straightforward in the internal language.
\end{proof}

\todo{The two above constructions, and the following ones, are actually left adjoints, freely closing (w.r.t.\ the \enquote{contains} ordering on families).  That’s very clear for the subobject-closure, but needs a little more thought for the product-closure using lists.  Is it worth saying it?  Probably — at least suppressing details.  Is this in any of the literature on universes, e.g.~\cite{streicher:universes-in-toposes,gratzer-shulman-sterling:strict-univs}? Make sure we credit that literature appropriately.}

\begin{proposition} \label{prop:regular-universe}
  For any family $q : \Ut \to \Ub$ in a topos, there is a simple universe  containing $\Ub$.
\end{proposition}

\begin{proof}
  First replace $\Ub$ by $q + {!_\NNO} : \Ut + \NNO \to \Ub + 1$, to include $\NNO$.
  (This can be skipped if $\Ub$ already contains $\NNO$.)
  Then apply the two preceding propositions in turn: add products, and then add subobjects, preserving the product-closure.
\end{proof}

\begin{proposition} \label{prop:universe-with-powerobjects}
  For any family $q : \Ut \to \Ub$ in a topos, there is a simple universe containing $\Ub$ and with power-objects for it.
\end{proposition}

\begin{proof}
  Take the slice-wise power-object $\pow_{\cE/\Ub}(\Ut) \to \Ub$; this tautologically has power-objects for $\Ub$.
  Now apply the previous proposition to this family, extending it to a simple universe.
  By construction, this has power-objects for $\Ub$; but it also contains $\Ub$, by closure under subobjects, since any object embeds into its own power-object as singletons.
\end{proof}

\begin{corollary} \label{cor:r-universe-in-topos}
  Any family in a topos is contained in some $r$-universe, for each $r \in \N$.
\end{corollary}

\begin{proof}
  Starting from the given family $\Ub_0$, we repeatedly apply \cref{prop:universe-with-powerobjects}, together with \cref{prop:universe-replace-mono} to keep the inclusions mono, yielding a sequence of simple universes $\Ub_0 \injto \Ub_1 \injto \cdots$, each with power-objects for the previous.
  Cutting off at stage $r$ evidently forms an $r$-universe.
\end{proof}

\begin{corollary} \label{cor:internal-sub-r-topos}
  Any family in a topos is contained in some internal sub-$r$-topos, for each $r \in \N$. \qed
\end{corollary}

\begin{proposition} \label{prop:standard-interpretation-internal-r-topos}
  Let $\cE$ be any topos.
  Then for each $r \in \N$, there is an essentially unique $r$-logical functor of $\cE$-indexed $r$-toposes $\HOLinterp : \freeTop[r][\cE] \to \cE$ (giving $\cE$ its self-indexing and full ranking), which we call the \defemph{standard interpretation} of $\freeTop[r][\cE]$.
\end{proposition}

\begin{proof}
  \Cref{cor:internal-sub-r-topos} gives an internal sub-$r$-topos $\smE \to \cE$.
  Composing this with the $r$-logical functor $\freeTop[r][\cE] \to \smE$ supplied by initiality yields an $r$-logical functor $\HOLinterp : \freeTop[r][\cE] \to \cE$ as desired.

  Essential uniqueness follows as in \cref{prop:free-topos-is-bi-initial,prop:free-r-topos-is-bi-initial}, requiring just the additional observation that for $r$-logical $F_0, F_1 : \freeTop[r][\cE] \to \cE$, the iso-comma $\isocomma{F_0}{F_1}$ is $\cE$-small, by cartesian closure of $\cE$.
  \comment{If referee asks for more here: Give indexed iso-comma $r$-topos above, and note it’s internal if the sides are internal and the vertex locally small.}
\end{proof}

\comment{Question: in weaker setting than a topos — e.g. a general AU, or $r$-topos with AU ranks — can the/a standard interpretation exist without being essentially unique?}

\begin{remark}
  As expected for reflection principles, \cref{prop:standard-interpretation-internal-r-topos} cannot be extended to the internal free \emph{topos}: Gödel incompleteness implies that there is a non-degenerate topos $\cG$ whose internal free topos $\freeTop[][\cG]$ is degenerate, and hence cannot admit any logical functor to $\cG$.
\end{remark}

\begin{remark} \label{rmk:universes-for-aus}
  If the ranks of an $r$-topos are assumed not just regular but in fact arithmetic universes, as suggested in \cref{rmk:what-if-ranks-aus}, then the universes of this section need to be closed under sums and list objects.
  This complicates the closure of a family $\Ut \to \Ub$ under all required operations (\cref{prop:regular-universe}): the closure is no longer simply \enquote{subobjects of formal products from the family}, easily indexed by $\pow(\List(\Ub))$, but must interleave formal \enquote{list object} and \enquote{sum} operations with the products, and hence be indexed by something like the power-object of a suitable $W$-type.
  This can certainly still be done, but requires substantially more work, especially to describe the resulting total space.
  
  On the flip side, with our current approach of regular ranks, the ranks of the universes constructed are unclear.
  Assuming ranks are AU’s should mean that the closure/successor constructions of \cref{prop:regular-universe,prop:universe-with-powerobjects} only increase rank by one or two, and hence allow the standard interpretation of an $r$-topos in a $(2r+1)$- or even $(r+1)$-topos.
  \comment{Would be good to check the bound here, and make our claim less equivocal if possible.}
\end{remark}

\subsection{Projectivity theorems} \label{sec:projectivity-thms}

With the internal standard interpretation of $\IHOL[r]$ in $\IHOL$ assembled, we are now ready to draw the desired applications to projectivity principles.

\begin{theorem}[Internal Freyd gluing]
  \label{thm:freyd-gluing-internal-r-topos}
  In any topos $\cE$, the terminal object of the internal free $r$-topos $\freeTop[r][\cE]$ is (indexed-) projective.
\end{theorem}

\begin{proof}
  We adapt the proof of \cref{prop:freyd-gluing-topos} once again, with just a little extra work this time for the internalisation.
  
  By \cref{cor:internal-sub-r-topos}, $\cE$ has some internal full sub-$r$-topos $\smE'$ including the family $\mor \freeTop[r][\cE] \to (\ob \freeTop[r][\cE])^2$ of hom-sets of $\freeTop[r][\cE]$, and hence such that the global sections functor $\Gamma : \freeTop[r][\cE] \to \cE$ factors as $\Gamma' : \freeTop[r][\cE] \to \smE' \injto \cE$.

  Now by \cref{prop:internal-artin-gluing}, the comma category $\comma{\smE'}{\Gamma'}$ is an internal $r$-topos, and its first projection functor $\freeTop[r][\cE]$ is strictly $r$-logical;
  so it admits an interpretation functor $\freeTop[r][\cE]$.
  Composing with the ($r$-logical, strictly over $\freeTop[r][\cE]$) inclusion $\comma{\smE'}{\Gamma'} \to \comma{\cE}{\Gamma}$ gives an $r$-logical strict splitting of the projection $\comma{\cE}{\Gamma} \to \freeTop[r][\cE]$:
  \[ \begin{tikzcd}[column sep={5em,between origins}]
    & \comma{\smE'}{\Gamma'} \ar[d,"P_0"] \ar[r,inj]
    & \comma{\cE}{\Gamma} \ar[d,"P_0"] \ar[r,"P_1"]
    & \cE \ar[d,"\id"] \ar[dl,Rightarrow,shorten=6mm]
    \\ \freeTop[r][\cE]
          \ar[ur,"\HOLinterp",dashed]
          \ar[r,"\id"]
     & \freeTop[r][\cE] \ar[r,"\id"]
     & \freeTop[r][\cE] \ar[r,"\Gamma"]
     & \cE
    \end{tikzcd}
  \]
  
  Now for any $X \in \cE$ and cover $A \coverto 1$ in $(\freeTop[r][\cE])_X$ (i.e.~map $A : X \to \EATinterp[ x \of {\ob} \suchthat x \coverto 1][\freeTop[r][\cE]])$, its interpretation in $\comma{\cE}{\Gamma}$ amounts (by 
  \cref{prop:commas}\zcref[noname]{item:comma-regular}) to data in $(\freeTop[r][\cE])_X$ and $\cE/X$ of the form
  \[ \begin{tikzcd}
    A \ar[d,->>] & & \HOLinterp[A]_1 \ar[r] \ar[d,->>] & \Gamma(A) \ar[d]
  \\ 1 & & \mathllap{1 = {}}\HOLinterp[1]_1 \ar[r] & \Gamma(1) \mathrlap{{}=1}
    \end{tikzcd}
  \]
  In $\cE$ itself, this amounts to:
  \[ \begin{tikzcd}
    & \EATinterp[ x \of {\ob} \suchthat x \coverto 1][\freeTop[r][\cE]] \ar[d]
    & \HOLinterp[A]_1 \ar[r] \ar[d,->>] & A^*\EATinterp[ x, s \suchthat s : 1 \to x][\freeTop[r][\cE]] \ar[d]
  \\ X \ar[ur,"A"] \ar[r,"1"] & \ob \freeTop[r][\cE]
    & X \ar[r,"\id"] & X
    \end{tikzcd}
  \]
   This exhibits a global section of $A$ after reindexing along the cover $\HOLinterp[A]_1 \coverto X$, just as required for projectivity by \cref{prop:indexed-proj-equivalent-forms}.
\end{proof}

\comment{Originally felt somehow like this should be a translation of the proof of \cref{prop:freyd-gluing-r-topos} under the stack semantics of \cite{shulman:stack-semantics-and-comparison}; but now hard to see what can be said about this. Probably ignore.}

We need just one last ingredient to give the main result:

\begin{definition}
  Let $\cE$ be an arithmetic universe.
  By the \defemph{internalisation of syntax to $\cE$}, we mean the unique $\cE$-indexed $r$-logical functor $\mathquote{-} : \freeTop[r] \to \freeTop[r][\cE]$ (provided by \cref{prop:free-r-topos-indexed-initial}).
\end{definition}

\begin{proposition}
  \label{prop:extend-interp-to-internal-free}
  Let $\cE$ be a topos.  Then internalisation of syntax to $\cE$, followed by the internal standard interpretation, yields the interpretation of actual syntax into $\cE$.
  That is, the following triangle of $\cE$-indexed $r$-logical functors commutes up to canonical isomorphism:

  \[ \begin{tikzcd}[row sep=small]
    \freeTop[r] \ar[dr,"\mathquote{-}"'] \ar[drrr,bend left=20,"\HOLinterp"{description,pos=0.45}] & & \\
    \ & \freeTop[r][\cE] \ar[rr,"\HOLinterp"{pos=0.45}] & & \cE
  \end{tikzcd} \]
\end{proposition}

\begin{proof}
  Immediate by bi-initiality of $\freeTop[r]$.
\end{proof}

\begin{theorem} \label{thm:n-is-projective-in-f}
  The natural numbers object of the free topos $\freeTop$ is projective.
\end{theorem}

\begin{proof} 
  By \cref{prop:free-ranked-topos-equiv-free-topos}, it is equivalent to show this for the free $\infty$-ranked topos $\freeTop[\infty]$.
  By compactness (\cref{prop:free-inf-top-filtered-colim,cor:covers-in-filtered-colim}), any cover over $\NNO$ in $\freeTop[\infty]$ is the image of such a cover in $\freeTop[r]$ for some finite $r$.
  So it suffices to show: for any $r \in \N$, any cover $e : A \coverto \NNO$ in $\freeTop[r]$, and any topos $\cE$, $\HOLinterp[e]^\cE$ has a section in $\cE$.
  
  Given such $r$, $e$, $\cE$, consider the internalisation $\mathquote{e} : \mathquote{A} \coverto \mathquote{N}$ in $\freeTop[r][\cE]$; this is a cover since $\mathquote{-}$ is regular.
  The projectivity of $1$ in $\freeTop[r][\cE]$, \cref{thm:freyd-gluing-internal-r-topos}, tells us by way of \cref{prop:indexed-proj-equivalent-forms}\zcref[noname]{item:ind-proj-representable} that taking $\cE$-enriched global sections preserves covers;
  so $\freeTop[r][\cE](\mathquote{e}) : \indhom[\cE]{{\freeTop[r][\cE]}}{}{1}{\mathquote{A}} \coverto \indhom[\cE]{{\freeTop[r][\cE]}}{}{1}{\mathquote{N}}$ is a cover in $\cE$.
  Pulling this back along the numeral map $\Numeral : N \to \indhom[\cE]{{\freeTop[r][\cE]}}{}{1}{\mathquote{N}}$ gives the square
  \[
    \begin{tikzcd}
      \Numeral^*( \indhom[\cE]{{\freeTop[r][\cE]}}{}{1}{\mathquote{A}}) \ar[r] \ar[d,->>] \ar[dr,drpb]
      & \indhom[\cE]{{\freeTop[r][\cE]}}{}{1}{\mathquote{A}} \dar[->>] \\
      N \rar["\Numeral"]
      & \indhom[\cE]{{\freeTop[r][\cE]}}{}{1}{\mathquote{\NNO}}
    \end{tikzcd}
  \]
  The whole pullback lifts to $\exlex{(\Pred \cE)}$, so $\Numeral^*(\indhom[\cE]{{\freeTop[r][\cE]}}{}{1}{\mathquote{A}})$ is enumerable and we may take by \cref{prop:enum-to-strenum-split-image} some section of the left-hand map (picking the Gödel-number-minimal witnessing term);
  equivalently, a map $s : N \to \indhom[\cE]{{\freeTop[r][\cE]}}{}{1}{\mathquote{A}} $ over $\indhom[\cE]{{\freeTop[r][\cE]}}{}{1}{\mathquote{\NNO}} $.

  This then yields the desired section of $\HOLinterp[e] : \HOLinterp[A] \to N$ in $\cE$, just by applying the standard interpretation functor $\HOLinterp : \freeTop[r][\cE] \to \cE$ of \cref{prop:standard-interpretation-internal-r-topos} and the internalisation-interpretation isomorphism of \cref{prop:extend-interp-to-internal-free}:
  \[
    \begin{tikzcd}[column sep = small]
      & \indhom[\cE]{{\freeTop[r][\cE]}}{}{1}{\mathquote{A}} \ar[r] \ar[d,cover]
      & \indhom[\cE]{\cE}{1}{\HOLinterp[1]}{\HOLinterp[\mathquote{A}]}  \ar[r,"\iso"] \ar[d,cover]
      & \HOLinterp[\mathquote{A}] \ar[d,cover] \ar[r,"\iso"]
      & \HOLinterp[A] \ar[d,cover]
    \\
      N \ar[ur, "s", dashed] \rar
      & \indhom[\cE]{{\freeTop[r][\cE]}}{}{1}{\mathquote{N}} \rar
      & \indhom[\cE]{\cE}{1}{\HOLinterp[1]}{\HOLinterp[\mathquote{N}]} \rar["\iso"]
      & \HOLinterp[\mathquote{N}] \rar["\iso"]
      & N
    \end{tikzcd}
  \]
  The bottom composite is $\id_N$ since each step is an $(0,S)$-algebra map.
\end{proof}

\begin{corollary} \label{cor:ihol-rcc}
  Intuitionistic higher-order logic admits the rule of countable choice:
  \[ \inferrule*[right=$\RCC$]
    { \proves \lforall{n \of N} \lexists{x \of X} \varphi(n,x) }
    { \proves \lexists{f \of X^N} \lforall{n \of N} \varphi(n,f(n)) } 
  \]
  
  Indeed, if $\IHOL$ proves \enquote{$\lforall{n \of N} \lexists{x \of X} \varphi(n,x)$}, there is some $\IHOL$-definable function $f \of X^N$ for which $\IHOL$ proves \enquote{$\lforall{n \of N} \varphi(n,f(n))$}.
\end{corollary}

\begin{proof}
  Provability of $\lforall{n \of N} \lexists{x \of X} \varphi(n,x)$ amounts precisely to the projection $\pi_0 : \HOLinterp[ n\of N,\, x \of X \mid \varphi(n,x) ] \to N$ being a cover in $\freeTop$ \parencite[II.6.1(b)]{lambek-scott:intro}.

  Projectivity of $\NNO$ provides a section $s$ for $\pi_0$;
  but by conservativity of the interpretation \parencite[II.14.3]{lambek-scott:intro}, $s$ amounts precisely to a definable function $f : N \to X$ such that $\IHOL \proves \lforall{n \of N} \varphi(n,f(x))$, as required for $\RCC$.
\end{proof}

In fact, this may be strengthened to a rule of \emph{dependent} choice:
\[ \inferrule*[right=$\RDC$]
  { \proves \lexists{x \of X} \varphi(x) 
  \\ \proves \lforall{x \of X} \left( \varphi(x) \Imp \lexists{y \of X} \left( \varphi(y) \land \rho(x,y) \right) \right) }
  { \proves \lexists{f \of X^N} \lforall{n \of N} \left( \varphi(f(n)) \land \rho(f(n),f(n+1))\right) } 
\]

Categorically, this is most clearly formulated in terms of graphs.

\begin{definition} \leavevmode \nopagebreak
  \begin{enumerate}
  \item A \defemph{graph} $\smG$ in a category $\cE$ is just a parallel pair $s, t : G_1 \parpair{}{} G_0$.

  \item A graph is \defemph{simple} if $s,t$ are jointly monic, and (assuming $\cE$ regular) \defemph{total} if $s : G_1 \coverto G_0$ and $G_0 \coverto 1$ are covers.

  \item Assuming a natural numbers object $(\NNO,0,S)$ in $\cE$, a \defemph{branch} in $\smG$ is a pair of maps $f_0 : \NNO \to G_0$, $f_1 : \NNO \to G_1$ with $s f_1 = f_0$, $t f_1 = f_0 S$.
  \end{enumerate}
\end{definition}

When $\cE$ interprets sufficient logic, the premises of $\RDC$ precisely describe a simple total graph in $\cE$
\[ \HOLinterp[ x,y \of X \suchthat \varphi(x) \land \varphi(y) \land \rho(x,y)] \ \parpair[2em]{}[][cover,yshift=0.2ex]{}[][yshift=-0.2ex] \ 
\HOLinterp[ x \of X \suchthat \varphi(x) ] \coverto[\quad] 1, \] 
and its conclusion asserts (internal) existence of a branch.
As with $\RCC$, we get the slightly stronger conclusion (equivalent since $1$ is projective) of external/global existence:

\begin{theorem} \label{thm:total-graph-branch-in-f}
  Every total graph in $\freeTop$ has a branch; so $\freeTop$ validates $\RDC$.
\end{theorem}

\begin{proof}
  By compactness as in \cref{thm:n-is-projective-in-f}, any total graph in $\freeTop$ lifts to some $\freeTop[r]$,
  so it suffices to show: for any $r \in \N$, any total graph $\smG$ in $\freeTop[r]$, and any topos $\cE$, $\HOLinterp[ \smG ]^\cE$ has a branch in $\cE$.

  Again, we next internalise $\smG$ to a total graph $\mathquote{\smG}$ in $\freeTop[r][\cE]$.
  Taking ($\cE$-enriched) global sections, the resulting graph $\indhom[\cE]{{\freeTop[r][\cE]}}{}{1}{\mathquote{\smG}}$ is still total by $\cE$-indexed projectivity of $1$ in $\freeTop[r][\cE]$.
  
  We know $\indhom[\cE]{{\freeTop[r][\cE]}}{}{1}{\mathquote{\smG}}$ lifts to $\exlex{(\Pred \cE)}$, so denoting its canonical cover by $p : \Gbar_0 \coverto \indhom[\cE]{{\freeTop[r][\cE]}}{}{1}{\mathquote{G_0}}$, we can pull the arrows $\indhom[\cE]{{\freeTop[r][\cE]}}{}{1}{\mathquote{G_1}}$ back along $p \times p$ to obtain a graph $\smGbar$, again total and with $\Gbar_1$ enumerable, but now also with $\Gbar_0$ (the object of \enquote{terms of type $\mathquote{G_0}$}) strictly enumerable.
  
  The covers $\overline{s} : \Gbar_1 \coverto \Gbar_0$ and $\Gbar_0 \coverto 1$ now admit by \cref{prop:enum-to-strenum-split-image} some splittings $k$, $a$;
  these give a branch in $\smGbar$ 
  (with $f_0 : \NNO \to \Gbar_0$ specified by $f_0 0 = a$, $f_0 S = \bar{t}kf_0$),
  and mapping this across through the internal standard interpretation gives the required branch in $\HOLinterp[\smG][\cE]$:
  \[
    \begin{tikzcd}[baseline=(\tikzcdmatrixname-3-1.base)]
        \;\! \Gbar_1\!\!\: \ar[dr,drpb] \ar[r] \ar[d,cover,"\bar{s}" description] \ar[d,shift left=2,"\bar{t}"]
      & \indhom[\cE]{{\freeTop[r][\cE]}}{}{1}{\mathquote{G_1}} \ar[d,cover,shift right] \ar[d,shift left] \ar[r,"\HOLinterp"]
      & \indhom[\cE]{\cE}{}{1}{\HOLinterp[\mathquote{G_1}]} \ar[r,"\iso"] \ar[d,cover,shift right] \ar[d,shift left]
      & \HOLinterp[G_1][\cE]  \ar[d,cover,shift right,"s"'] \ar[d,shift left,"t"]
      \\ \, \Gbar_0\! \ar[r,cover] \ar[d,cover] \ar[u,bend left,dashed,"k"]
      & \indhom[\cE]{{\freeTop[r][\cE]}}{}{1}{\mathquote{G_0}} \ar[d,cover] \ar[r,"\HOLinterp"]
      & \indhom[\cE]{\cE}{}{1}{\HOLinterp[\mathquote{G_0}]} \ar[r,"\iso"] \ar[d,cover]
      & \HOLinterp[G_0][\cE]  \ar[d,cover]
      \\ 1 \ar[u,bend left,dashed,"a"]
      & 1
      & 1
      & 1
    \end{tikzcd}
    \qedhere \]
\end{proof}

\begin{corollary} \label{cor:ihol-rdc}
  The rule of dependent choice $\RDC$ is admissible for $\IHOL$.
  Indeed, whenever $\IHOL$ proves \enquote{$\lexists{x \of X} \varphi(x)$} and \enquote{$\lforall{x \of X} ( \varphi(x) \Imp \allowbreak \lexists{y \of X} \allowbreak ( \varphi(y) \land \rho(x,y) ))$}, there is some $\IHOL$-definable function $f : X^\NNO$ for which $\IHOL$ proves \enquote{$\lforall{n \of N} \allowbreak \left( \varphi(f(n)) \land \rho(f(n),f(n+1))\right)$}. \qed
\end{corollary}

\subsection{Comparison with Makkai’s presentation} \label{sec:comparison-makkai}

The present development differs in several respects from the treatment in Makkai’s manuscript \cite{makkai:n-is-projective}.

Most significantly, Makkai makes use throughout of the correspondence between $\IHOL$ (presented in a traditional syntax) and elementary toposes, while we work purely categorically, referring to $\IHOL$ occasionally for intuition, but never formally until the final \cref{cor:ihol-rcc}.
For instance, Makkai defines an $r$-topos as a Heyting category with subobject classifier that admits interpretations of all $\IHOL$-types of rank $\leq r$: he ranks the syntax, not its categorical interpretation.
Similarly, he presents free ($r$\nobreakdash-)toposes using the syntax of $\IHOL$ (along similar lines to \cite[\textsection \textsection 11--12]{lambek-scott:intro}), and their internalisations using the Gödel-number encoding of this syntax.
By contrast, we avoid encoding details by appealing to the general existence of free structures in arithmetic universes,
and correspondingly using a finitely-presented essentially algebraic notion of $r$-topos.

A more inessential difference, but perhaps of interest, is that where we have used indexed/fibred categories to mediate between external and internal categories, Makkai uses \emph{enriched} categories instead.

\subsection{History and related work} \label{sec:history}

The early history of work on $\RCC$ is somewhat murky, with more folklore than published literature; the following is the best understanding we have been able to reach, based in part on correspondence with Phil Scott and conversations with Michael Makkai.
For now we survey this just in outline, deferring discussion of technicalities to \cref{sec:comparison-approaches}

The first proof of admissibility of $\RCC$ was for second-order arithmetic ($\HAS$), in Troelstra \cite[Thm.~4.5.12]{troelstra:metamathematical-investigation},
couched in traditional proof-theoretic language
(and with a technical error corrected by Hayashi \cite[App.~3]{hayashi:on-derived-rules}).%
\footnote{This erratum merits a little explanation, as the descriptions in \cite{hayashi:on-derived-rules} and \cite{troelstra-2018:corrections} are slightly obscure: they strengthen Lemma~4.5.8 of \cite{troelstra:metamathematical-investigation}, but do not mention where the strengthening is required.
The issue is that Lemmas~4.5.8 and 4.5.9 can each be stated in two forms: with an external quantification over actual formulas, or an internal quantification over (Gödel-codes of) formulas in $\HAS$.
\cite{troelstra:metamathematical-investigation} gives both in their external forms, which indeed suffice for subsequent applications there.
However, the proof of Lemma~4.5.9 involves an internal induction over (codes of) proofs, so needs its statement to range over arbitrary internal formulas, and thus requires the stronger, internal version of Lemma~4.5.8.}

Following the splash of Freyd’s elegant categorical proof of $\EP$ and $\DP$ for toposes, there was a general interest in extending categorical methods to other proof-theoretic results;
in particular, proposals for adapting $\RCC$ to $\IHOL$ were presented at least by Joyal, Makkai, and Friedman and Scedrov.
However, this work of Joyal and Makkai never made it into print, while that of Friedman and Scedrov appeared as \cite{friedman-scedrov:set-existence-property}, including several related results but not $\RCC$.

On the non-categorical side, it has been suggested to us that $\RCC$ for some second- or higher-order system appears in work of Hayashi, but we have been unable to locate it there.
\Cite{hayashi:on-derived-rules} builds on and corrects the methods of \cite[\textsection 4.5]{troelstra:metamathematical-investigation}, giving various derived rules for a variant of $\HAS$, using approaches somewhat analogous to those of the present proof (formalising $\EP$ and satisfaction for fragments of $\HAS$ within $\HAS$), and outlining how to adapt the results to a higher-order system; but it does not consider $\RCC$ or similar principles.
It thus seems likely to us that Hayashi, too, may have presented a proof of $\RCC$ during this period that was never published.

Subsequent work using the technique of \emph{realizability with truth} has shown admissibility of  the full rule of choice for various constructive systems, mainly $\HAomega$ and close variants \cite[Theorem 3.8(iv)]{troelstra:realizability}, including just the \emph{finite} types (that is, iterated exponentials of $\NNO$).
Recent work of Frittaion, Nemoto, and Rathjen \cite{frittaion-nemoto-rathjen:choice-and-independence-rules} extends this method to constructive set theories, but retains the type restriction in the rule considered: Thm.~8.2 there shows that a range of theories including CZF and IZF admit the \enquote{rule of choice at finite types}:
\[ \inferrule*[right=$\RCFT$,left=\textnormal{($\sigma$, $\tau$ finite types)}]
  { \proves \lforall{n \of \sigma} \lexists{x \of \tau} \varphi(n,x) }
  { \proves \lexists{f \of \tau^\sigma} \lforall{n \of N} \varphi(n,f(n)) } 
\]
These results are thus neither more or less general than $\RCC$ for systems with power-types/-sets: the domain type of the choice is more general, but the target type is restricted, and in particular cannot involve power-types.

\subsection{Comparison with proof-theoretic presentations} \label{sec:comparison-approaches}
In the introduction we outlined the argument in terms of $\IHOL$; then we actually proved it entirely categorically, and extracted the $\IHOL$ version as a corollary.
As usual, the categorical treatment could be compiled down
into a more traditionally proof-theoretic treatment working syntactically in
$\IHOL$, along the lines of the introductory sketch. 

In such a development, our reflection theorems (\cref{sec:reflection-results}) would turn out essentially like Lemmas 4.5.8–9 of \cite{troelstra:metamathematical-investigation}, and our endgame much like the proof of Thm.~4.5.12 there (all adapted from $\HAS$ to $\HAH$/$\IHOL$).
The most substantial difference from Troelstra/Hayashi’s approach is that they obtain $\EP$ by normalisation techniques, where we used Freyd’s gluing argument.
As shown in Scedrov and Scott \cite{scedrov-scott:friedman-slash-freyd-covers}, Freyd’s gluing topos corresponds on the syntactic side to a variant of the Kleene slash due to Friedman%
\footnote{In fact many normalisation arguments can also be fruitfully analysed as gluing constructions \cite{fiore:semantic-analysis-of-nbe}, as can realizability with truth \cite[\textsection 4.6.1]{van-oosten:realizability}.}
; our proof thus corresponds to suitable internalisations of the Friedman slash for fragments of $\IHOL$/$\HAH$.

How does such a treatment compare with the present version?
Purely proof-theoretic approaches can be more elementary, and in many respects intuitively and expositorily simpler.
However, the bureaucratic encoding required for internalisation of syntax, and the subsequent internalisation of arguments, is typically given by hand just as required — hard to write and worse to read, so almost always suppressed. 
As Troelstra writes \cite[4.5.12]{troelstra:metamathematical-investigation}: \enquote{We sketch the argument: full details are long and tedious, and hardly instructive.}

Categorical presentations cannot fully avoid these encoding details, but they allow
them to be organised more uniformly, and given in clearer generality.
So the tedious work can be done once and for all, with a precise and broad scope of applicability.

Overall, the picture is a common one:
The categorical framework entails some extra overhead, not so much of technical work as of intuition. In return, however, it permits a more modular presentation, with more of the
work packaged into general, broadly applicable results, and less ad hoc
calculation.

Most fruitful of all, however, is the flexibility to pass freely back and forth, with the categorical framework to handle higher-level structural results, but dropping back into syntax for more elementary work as convenient, as we did for the low-level constructions of \cref{sec:universes,sec:internal-free-models}.

\begin{private}
  \addtocontents{toc}{\protect\setcounter{tocdepth}{1}}
  \appendix

  \input{appendix-literature.tex}

\section{Overview of notations}

\comment{Primarily for private reference — probably delete or suppress before publication.}

Here we lay out, for reference and overview, the notational conventions we aim to follow in the paper.

\begin{itemize}
\item $2$-categories of large categories: $\CAT$, $\REG$, $\LOG$, $\LOG[r]$; of indexed cats, $\CAT[\cB]$, etc.
\item $2$-categories of small categories: $\Cat$, $\Log$, $\Log[r]$; of indexed cats, $\Cat[\cE]$ etc.; of internal cats, $\intCat[\cE]$ etc.
\item Specific large ($1$-)categories, e.g. of small or internal algebraic structures: $\strCat[\cE]$, $\strLog[r]$, $\Mod[\cE]{\thT}$
\item Arbitrary potentially-large categories: $\cE$ (typically when logically-structured), $\cC$
\item Specific logical free/syntactic categories (small, but nonetheless important citizens of the world of possibly-large categories): $\freeTop[r][\cE]$, $\syncat{\thT}$
\item Arbitrary small categories: $\smE$ (esp.\ if logical), $\smC$
\item Arbitrary internal categories: $\smE$ (esp.\ if logical), $\smC$
\end{itemize}

\comment{The Elephant uses fraktur for all $2$-categories. We roughly follow this, but use calligraphic instead: both are a bit ugly, but Fraktur in most fonts is less readable.}

Interpretations of logical languages in categories:
\begin{itemize}
\item $\IHOL$ (or similar) in a topos: $\HOLinterp[ X^\NNO ][\cE]$.
\item Internal language of an AU: $\AUinterp[ \NNO \times \NNO ][\cE]$. \comment{Should we distinguish this notationally from the previous? We use this fairly informally/flexibly, whereas that one is a quite formal item.}
\item EAT in a model: $\EATinterp[ X ][M]$.
\end{itemize}

Indexed hom-objects: $\indhom{\cE}{X}{A}{B}$.

\comment{Alternatives:
$\indhomwithHom{\cE}{\indC}{X}{A}{B}$, or 
$\indhomwithunderline{\cE}{\indC}{X}{A}{B}$, or 
$\indhomwithcatsubscript{\cE}{\indC}{X}{A}{B}$, or 
$\indhomwithboldparens{\cE}{\indC}{X}{A}{B}$…}

\begin{longcomment}
  Arrows testing: $A \to B$, $A \rightarrow B$, $A \injto B$, $A \hookrightarrow B$, $A \coverto B$ $A \twoheadrightarrow B$.

  Parallel pairs:
  $ A \parpair[2em]{s}[xshift=-0.2em][cover]{t} B$,
  $ A \parpair{}{} B$

  \filbreak 
  Check how inline arrows affect line spacing: testing paragraph and enough test to fill out a line of words text filling out the line like a long and rather stupid sentence.  Now some text again to get onto the next line and then $A \to B$ let’s see how it raises the line height compared to the other lines, and what about the depth as well.
  Now what about if we put a parallel pair in in the style above:
  $A \parpair{s}{t} B$
 what does that look like?  Good; not too much effect on the line spacing now!

\end{longcomment}

\begin{longcomment}
  Fonts testing: \newcommand{\testcatnames}[1]{$#1{CAT}$, $#1{Cat}$, $#1{REG}$, $#1{Reg}$, $#1{LOG}$, $#1{Log}$}
  \begin{itemize}
  \item \testcatnames{\mathbf}
  \item \testcatnames{\mathscr}
  \item \testcatnames{\mathcalit}
  \item \testcatnames{\mathbfcalit}
  \item \testcatnames{\mathfrak}
  \end{itemize}

  Testing scale-matching between fonts/alphabets: \\
  \begingroup
  \noindent
  \rlap{\color{red}\vrule width\textwidth height.15pt depth.15pt}%
  \rlap{\color{red}\vrule width\textwidth
           height\dimexpr1ex+.15pt\relax depth-\dimexpr1ex-.15pt\relax}%
  \sbox0{A}%
  \rlap{\color{red}\vrule width\textwidth
           height\dimexpr\ht0+.15pt\relax depth-\dimexpr\ht0-.15pt\relax}%
  \newcommand{\testletters}{ABQabxgf}
  \testletters%
  $\testletters$%
  $\mathcalit{ABQabxgf}$%
  $\mathfrak{\testletters}$
  \textsf{\testletters}%
  \texttt{\testletters}%
  \textsc{\testletters}%
  \endgroup

  Kerning testing, for mathcal category names:
  \newcommand{\testkernrow}[1]{
    #1\textrm{mu}
    & \mathcalit{C{\mkern#1mu} AT}
    & \mathcalit{C{\mkern#1mu} at}
    & \mathcalit{L{\mkern#1mu} og}
    & \mathcalit{R{\mkern#1mu} eg}
    & \mathcalit{L{\mkern#1mu} ex}
    & \CAT{\mkern#1mu}_{\cE}
    & \mathcalit{LE{\mkern#1mu}XT}
    & \mathcalit{PRET{\mkern#1mu}OP}
  }
  \[ \begin{array}{rccccccccc}
    \testkernrow{0} \\
    \testkernrow{-0.5} \\
    \testkernrow{-1} \\
    \testkernrow{-1.5} \\
    \testkernrow{-2} \\
    \testkernrow{-2.5} \\
    \testkernrow{-3} \\
    \testkernrow{-3.5} \\
    \testkernrow{-4} \\
    \hline \text{best} & 1.5? & 2? & 2? & 2? & 2? & -3.5? & -1? & -3? \\
     & \CAT & \Cat & \Log & \Reg & \Lex & \CAT[\cE] & \LEXT & \PRETOP
  \end{array}\]
\end{longcomment}

\end{private}


\printbibliography

\end{document}